\documentclass[11pt,reqno]{amsart}
\usepackage{ytableau,tikz,hyperref}
\usepackage{mathtools}
\usepackage{palatino}
\usepackage[left=3cm,right=3cm,top=3cm,bottom=3cm]{geometry}
\usepackage[shortlabels,inline]{enumitem}
\newcommand\characx{\mathbf{x}}
\newcommand{\key}{\kappa}
\newtheorem{theorem}{Theorem}
\newtheorem{remark}{Remark}
\newcommand\bremark{\begin{remark}\begin{upshape}}
\newcommand\eremark{\end{upshape}\end{remark}}
\newtheorem{proposition}{Proposition}
\newtheorem{corollary}{Corollary}
\newtheorem{lemma}{Lemma}
\newtheorem{definition}{Definition}
\newtheorem{example}{Example}
\DeclareMathOperator{\Tab}{Tab}
\DeclareMathOperator{\Mat}{Mat}
\DeclareMathOperator{\SVT}{SVT}
\DeclareMathOperator{\RPP}{RPP}
\DeclareMathOperator{\ceq}{ceq}
\DeclareMathOperator{\wt}{wt}
\DeclareMathOperator{\ex}{ex}
\DeclareMathOperator{\Read}{read}
\DeclareMathOperator{\rect}{rect}
\DeclareMathOperator{\sh}{shape}
\DeclareMathOperator{\SSDT}{SSDT}
\allowdisplaybreaks
\title[Key expansion of the flagged refined skew stable Grothendieck polynomial]{Key expansion of the flagged refined skew stable Grothendieck polynomial}
\author[]{Siddheswar Kundu}
\address{School of Mathematics, tata institute of fundamental research, mumbai, india}
\address{School of Mathematical Sciences, National Institute of Science Education and Research, Bhubaneswar, HBNI, P.O. Jatni, Khurda, Odisha, 752050, India
}
\email{kundusidhu96@gmail.com}
\begin{document}
\begin{abstract}
The flagged refined stable Grothendieck polynomials of skew shapes generalize several polynomials like stable Grothendieck polynomials, flagged skew Schur polynomials. In this paper, we provide a combinatorial expansion of the flagged refined skew stable Grothendieck polynomial in terms of key polynomials. We present this expansion by imposing a Demazure crystal structure on the set of flagged semi-standard set-valued tableaux of a given skew shape and a flag. We also provide expansions of the row-refined stable Grothendieck polynomials, the refined dual stable Grothendieck polynomials and the Schur P-functions in terms of stable Grothendieck polynomials $G_{\lambda}$ and in terms of dual stable Grothendieck polynomials $g_{\lambda}$.
\end{abstract}
\maketitle
\section{Introduction}
The Grassmannian $Gr(n,k)$ denotes the set of $k$-dimensional hyperplanes that lie within $\mathbb{C}^{n}$. Lascoux and Schützenberger \cite{Lascoux:G1,Lascoux:G2} introduced Grothendieck polynomials to describe the K-theory ring of the Grassmannian. Grothendieck polynomials can be viewed as a K-theory analogue of Schubert polynomials. Similar to the Schubert polynomial scenario, they are indexed by permutations in the symmetric group $S_n ,$ and by considering the stable limit of $n \rightarrow \infty ,$ Grothendieck polynomials are symmetric functions. Fomin and Kirillov \cite{Fomin:Yang-Baxter} commenced the study of $\beta$-Grothendieck polynomials with a parameter $\beta$, which transform into Schubert polynomials at $\beta=0$ and Grothendieck polynomials at $\beta=-1$. They also explored the stable Grothendieck polynomials with the parameter $\beta$.

For a partition $\lambda ,$ we denote the stable Grothendieck polynomials corresponding to the Grassmannian permutation $\pi_{\lambda}$ by $G_{\lambda}(\characx)$ (see \cite[\S2]{Buch:K-LR} for more details about $\pi_{\lambda}$). $G_{\lambda}$ can be regarded as a K-theory analogue of the Schur functions $s_{\lambda}$. $\{G_{\lambda}(\characx)\}$ indexed by partitions is a basis for (a completion of) the space of symmetric functions, see \cite{Lenart:Comb-aspects}. Buch \cite{Buch:K-LR} demonstrated that the stable Grothendieck polynomial $G_{\lambda}(\characx)$ is equal to a generating function for semi-standard set-valued tableaux of shape $\lambda$, which are a generalization of semi-standard Young tableaux. Following Buch's contributions \cite{Buch:K-LR}, many generalizations of $G_{\lambda}(\characx)$ have been thoroughly investigated from various perspectives. For example, Chan and Pflueger introduced the notion of the row-refined skew stable Grothendieck polynomial $RG_{\lambda/\mu}(\characx;\mathbf{t})$ based on the excess statistic in \cite[\S3]{Melody-Nathan}. In this paper, we study the flagged refined skew stable Grothendieck polynomial $G_{\lambda/\mu}(X_{\Phi}; \mathbf{t})$ (see \S\ref{Section 2} for the defintion), which is a common generalization of $RG_{\lambda/\mu}(\characx;\mathbf{t})$ and the flagged skew Schur polynomial $s_{\lambda/\mu}(X_{\Phi})$ in \cite{RS}. The flagged skew Grothendieck polynomials are also studied by Matsumura \cite{Matsumura:flagged}, see also \cite{Knutson-Miller-Yong}. 

Reiner and Shimozono \cite{RS} have given an expansion of the flagged skew Schur polynomials as a non-negative integral linear combinations of key polynomials (Demazure characters). A crystal theoretic statement of this result is proved in \cite[Theorem 3.11 \& Appendix]{KRSV}, namely, the set of flagged semi-standard Young tableaux is a disjoint union of Demazure crystals, which is further extended by the author \cite[Theorem 1]{Sidhu} by providing a Demazure crystal structure on the set of flagged reverse plane partitions. In this project, we provide a Demazure crystal structure on the set of all flagged semi-standard set-valued tableaux, denoted by $\SVT(\lambda/\mu,\Phi)$ in \S\ref{Section 2}, for a skew shape $\lambda/\mu$ and a flag $\Phi$. In order to do this, we prove that (in Theorem~\ref{theorem:main}), given a skew shape $\lambda/\mu$ and a flag $\Phi$, $\SVT_{\mathbf{e}}(\lambda/\mu, \Phi)$, i.e., the set of all flagged set-valued tableaux in $\SVT(\lambda/\mu, \Phi)$ with excess $\mathbf{e}$, admits a Demazure crystal structure. This theorem also generalizes that the set of flagged semi-standard Young tableaux is a disjoint union of Demazure crystals, see Corollary~\ref{corollary:SVT-tableau}. As a consequence, we obtain an expansion of $G_{\lambda/\mu}(X_{\Phi}; \mathbf{t})$ in terms of key polynomials (Corollary~\ref{corollary:main}) and an expansion of $RG_{\lambda/\mu}(\characx;\mathbf{t})$ in terms of Schur polynomials (Corollary~\ref{corollary:schur-SVT}).

The paper is organized as follows. In \S\ref{Section 2}, we review the definitions of semi-standard set-valued tableaux, key polynomials, the Burge correspondence and define the flagged refined skew stable Grothendieck polynomials. In \S\ref{Section 3}, we define the crystal structure on semi-standard set-valued tableaux of skew shape, for partition shape the crystal structure is defined in \cite{Travis:symmetric}. In \S\ref{Section 4}, we prove our main result (Theorem~\ref{theorem:main}, Remark~\ref{remark:main}), namely $\SVT(\lambda/\mu,\Phi)$ admits a Demazure crystal structure. In \S\ref{Section 5}, we provide expansions of the refined dual stable Grothendieck polynomials in terms of stable Grothendieck polynomials and in terms of dual stable Grothendieck polynomials and similarly for the row-refined skew stable Grothendieck polynomials and the Schur P-functions.
\section{Preliminaries}
\label{Section 2}
In this section, we review the definition of semi-standard set-valued tableaux. Then we define the flagged refined skew stable Grothendieck polynomials. We also review key polynomials, the Burge correspondence. 
\subsection{Partitions and skew Shapes}
A \emph{partition} is defined as a finite sequence of non-negative integers $\lambda=(\lambda_1,\lambda_2,\ldots)$ with $\lambda_1 \geq \lambda_2 \geq \cdots $. A \emph{hook partition} is a partition of the form $(a+1,1^{b}),$ where $a,b$ are non-negative integers. For $n \in \mathbb{N},$ $\mathcal{P}[n]$ denotes the set of all partitions with at most $n$ parts. The \emph{Young diagram} associated with the partition $\lambda$ is a collection of boxes that are top and left justified, where the $i^{th}$ row consists of $\lambda_i$ boxes. By abusing of notation, we denote the Young diagram of $\lambda$ once more by $\lambda$. Given two partitions $\lambda, \mu,$ we write $ \mu \subset \lambda$ if $ \mu_i \leq \lambda_i$ $\forall i \geq 1$. For partitions $\mu,\lambda$ such that $\mu \subset \lambda $, the \emph{skew shape} $\lambda / \mu$ is obtained by deleting the boxes of $\mu$ from those of $\lambda$. For example, see Figure~\ref{Figure 1}.
\begin{figure}
    \centering
\ytableausetup{nosmalltableaux}
\begin{ytableau} 
\none & \none & \null \\
 \none & \null & \null \\
\null & \null \\
\end{ytableau}
\caption{Skew shape $ (3,3,2)/(2,1)$}
\label{Figure 1}
\end{figure}
\subsection{Semi-standard set-valued tableaux}
Given two subsets $A,B$ of $\mathbb{N},$ we say $A \leq B$ if $\max(A) \leq \min(B)$ and $A < B$ if $\max(A) < \min(B)$.
A \emph{semi-standard set-valued tableau} \cite{Buch:K-LR} of shape $\lambda / \mu$ is a filling of the boxes of the shape $\lambda/\mu$ by non-empty subsets of $\mathbb{N}$ such that the rows are weakly increasing from left to right and the columns are strictly increasing from top to bottom. Note that a semi-standard Young tableau of shape $\lambda / \mu$ is a semi-standard set-valued tableau of the same shape where each box of the shape $\lambda/\mu$ is filled by a positive integer. When we write a set-valued tableau (SVT), we refer to a semi-standard set-valued tableau.

Let $\SVT_n(\lambda/\mu)$ $(\lambda,\mu \in \mathcal{P}[n])$ denote the set of all semi-standard set-valued tableaux of shape $\lambda/\mu$ with entries at most $n$. By $\Tab_n(\lambda/\mu)$, we mean the set of all semi-standard Young tableaux in $\SVT_n(\lambda/\mu)$. For $ S \in \SVT_n(\lambda/\mu)$, the \emph{weight} of $S$ is defined by $\wt(S):=(s_1,s_2,\ldots ,s_n),$ where $s_i$ is the number of occurences of $i$ in $S$ and the \emph{excess} of $S$ is $\ex(S):=(s'_1,s'_2,\ldots,s'_n),$ where $s'_i$ is the number of entries in $i^{th}$ row minus the number of boxes in that same row.
For example, see Figure~\ref{Figure 2}. 

\begin{figure}
     \centering
\begin{tikzpicture}[scale=0.95]
    \draw (0,0) -- (0,2) -- (3, 2) -- (3, 4) -- (4,4);
    \draw (4,4) -- (4,3) -- (1, 3) -- (1, 0) -- (0,0);
    \draw (0, 1) -- (2, 1) -- (2, 4) -- (3, 4);
\draw (0.5, 0.5) node {$3$};
\draw (0.5, 1.535) node {$1$};
\draw (1.5, 1.5) node {$2,3,4$};
\draw (1.49, 2.55) node {$1$};
\draw (2.5, 2.5) node {$3,4$};
\draw (2.5, 3.5) node {$1,2$};
\draw (3.55, 3.55) node {$2,3$};
\end{tikzpicture}
\caption{A semi-standard set-valued tableau of shape $(4,3,2,1)/(2,1)$, weight $(3,3,4,2)$,
excess $(2,1,2,0)$.}
     \label{Figure 2}
\end{figure}
\subsection{The flagged refined skew stable Grothendieck polynomials}
A flag $\Phi = (\Phi_1, \Phi_2, \ldots)$ is defined as a finite weakly increasing sequence of positive integers. For $n \in \mathbb{N},$ $\mathcal{F}[n] $ denotes the set of all flags $\Phi = (\Phi _1, \Phi _2, \ldots, \Phi_n)$ such that $\Phi_n = n$. Let $\Phi \in \mathcal{F}[n]$. Then we say a semi-standard set-valued tableau $T$ of shape $\lambda/\mu$ $(\lambda,\mu \in \mathcal{P}[n])$ respects flag $\Phi$ if the entries in $i^{th}$ row of $T$ is at most $\Phi_i$ for all $ 1 \leq i \leq n$. $ \SVT(\lambda/\mu,\Phi)$ denotes the set of all semi-standard set-valued tableaux of shape $\lambda/\mu$ that respects the flag $\Phi$. We define the \emph{flagged refined skew stable Grothendieck polynomial} $G_{\lambda/\mu}(X_\Phi; \mathbf{t})$ by 
$$ G_{\lambda/\mu}(X_\Phi; \mathbf{t}):= \displaystyle \sum _{T \in \SVT(\lambda/\mu,\Phi)} (-1)^{|\ex(T)|}\mathbf{t}^{\ex(T)}\characx^{\wt(T)} , $$ where
for $\alpha =(\alpha_1,\alpha_2,\ldots, \alpha_n) \in \mathbb{Z}^n _{+} ,\text{ we let }  \characx^{\alpha}=x_1^{\alpha_1}x_2^{\alpha_2}\cdots x_n^{\alpha_n},$ $ \mathbf{t}^{\alpha}=t_1^{\alpha_1}t_2^{\alpha_2}\cdots t_n^{\alpha_n}$ and $|\alpha|=\alpha_1+\alpha_2+\cdots+\alpha_n$ ($\mathbb{Z}^n _{+}$ is the set of all $n$-tuples of non-negative integers).

Let $\mathbf{1}=(1,1,\ldots,1), \mathbf{0}=(0,0,\ldots,0) \in \mathbb{Z}^n _{+}$. Then  
\begin{enumerate}
    \item At $t_1=\cdots=t_n=-\beta,$ $ G_{\lambda/\mu}(X_\Phi; \mathbf{t})$ reduces to the flagged skew Grothendieck polynomial $G_{\lambda/\mu,\mathbf{1}/\Phi}(x)$ (see \cite[\S4]{Matsumura:flagged}).
    \item $G_{\lambda/\mu}(X_\Phi; \mathbf{0}) $ is the flagged skew Schur polynomial 
    $s_{\lambda/\mu}(X_\Phi)$ in \cite{RS}, defined by, $$s_{\lambda/\mu}(X_\Phi):=\displaystyle\sum_{T \in \Tab(\lambda/\mu,\Phi)}\characx^{\wt(T)},$$ where $\Tab(\lambda/\mu,\Phi)$ is the set of all semi-standard Young tableaux in $\SVT(\lambda/\mu,\Phi)$.
    \item If $\Phi=(n,n,\ldots, n)$ then
\begin{itemize}
 \item $G_{\lambda/\mu}(X_\Phi; \mathbf{t})$ reduces to the row-refined skew stable Grothendieck polynomial $RG_{\lambda/\mu}(\characx;\mathbf{t}),$ see \cite[\S3]{Melody-Nathan}.
 \item $G_{\lambda/\mu}(X_\Phi; \mathbf{1})= G_{\lambda/\mu}(x_1,x_2,\ldots,x_n),$ the single stable Grothendieck polynomial (\cite[\S3]{Buch:K-LR}).
 \item $G_{\lambda/\mu}(X_\Phi; \mathbf{0})$ is the skew Schur polynomial $s_{\lambda/\mu}(\characx)
$.
\end{itemize} 
\end{enumerate}
\subsection{Key polynomials}Consider the polynomial ring $\mathbb{Z}[x_1,x_2,\dots,x_n]$. Then for $1 \leq i \leq n-1,$ the Demazure operators $T_i :\mathbb{Z}[x_1,x_2,\dots,x_n] \rightarrow \mathbb{Z}[x_1,x_2,\dots,x_n]$ 
are defined by: 
$$ T_i(f):= \frac{x_i\, f - \, x_{i+1}s_i(f)}{x_i - x_{i+1}},$$
where $s_i$ acts on $f$ by interchanging $x_i$ and $x_{i+1}$.    

Given a permutation $w \in S_n$, we define $$T_w :=  T_{i_1}T_{i_2} \cdots T_{i_k},$$ where $s_{i_1}s_{i_2} \cdots s_{i_k}$ is a reduced expression of $w$. Since $T_i$ satisfies the braid relations, $T_w$ does not rely on the reduced expression.

For any $\alpha \in \mathbb{Z}^n _{+} ,$ the \emph{key polynomial} is defined by $\key_{\alpha}:=T_w(\characx^{\alpha^{\dagger}}),$ where $\alpha^{\dagger}$ is the partition formed by sorting the parts of $\alpha$ into decreasing order and $w$ is any permutation in $S_n$ such that $w.\alpha^{\dagger}=\alpha$. Here, $w$ acts on $\alpha^{\dagger}$ by the usual left action of $S_n$ on $n$-tuples.

\subsection{The Burge correspondence}We write $[n]$ to denote the set $\{1,2,\dots,n\}$. Given $m,n \in \mathbb{N},$ $\Mat_{m \times n}(\mathbb{Z}_{+})$ is the set of all $m \times n$ matrices having non-negative integer entries. We associate a matrix $A=(a_{ij}) \in \Mat_{m \times n}(\mathbb{Z}_{+})$ to a biword $w_A$ as follows:
$$ w_{A}=\begin{bmatrix}         
i_t\hspace{0.2cm} \cdots \hspace{0.2cm}i_2 \hspace{0.2cm} i_1 \\
j_t \hspace{0.2cm} \cdots \hspace{0.2cm} j_2 \hspace{0.2cm}j_1 \\
           
\end{bmatrix}$$
so that for any pair $(i,j) \in [m] \times [n],$ there are $a_{ij}$ columns in $w_A$ equal to $
\begin{bmatrix}
   i \\
   j
\end{bmatrix}
$ and those are ordered as follows.
\begin{itemize}
    \item $i_t \geq \cdots \geq i_2 \geq i_1 \geq 1 $.
    \item $ i_{k+1} > i_k$ whenever $ j_{k+1}>j_k$.
\end{itemize}
\begin{example}
$ \text{Take} A=
\begin{pmatrix}
1 & 2 & 0\\
2 & 1 & 3
\end{pmatrix}
,$ then $w_A =
\begin{bmatrix}
2 & 2 & 2 & 2 & 2 & 2 & 1 & 1 & 1 \\
1 & 1 & 2 & 3 & 3 & 3 & 1 & 2 & 2 
\end{bmatrix}.$   
\end{example}
\begin{theorem}
\begin{upshape}
\cite[Appendix A, Proposition 2]{Fulton:yt} 
\end{upshape}
\label{Theorem:Burge}
The Burge correspondence gives a bijection between $\Mat_{m \times n}(\mathbb{Z}_{+})$ and the set of pairs $(P, Q)$, where $P, Q$ are both semi-standard Young tableaux of same shape and entries of $P, Q$ are in $[n],[m]$ respectively. We write $(w_A \rightarrow \emptyset)=(P, Q)$ if $A$ corresponds to $(P, Q)$. 
\end{theorem}
\section{Crystal structure on set-valued tableaux}
\label{Section 3}
In this section, we recall the notion of crystals of type $A_{n-1}$ and define an crystal structure on set-valued tableaux of skew shape extending the crystal structure for partition shape, given by Monical-Pechenik-Scrimshaw \cite[\S3]{Travis:symmetric}.   
\subsection{Crystals}A \emph{crystal} of type $A_{n-1}$ contains an underlying non-empty finite set $\mathcal{B}$ along with the maps 
$$ e_i,f_i :\mathcal{B} \rightarrow \mathcal{B} \sqcup \{0\} \ \ \text{for $i \in [n-1],$} $$
$$ \wt: \mathcal{B} \rightarrow \mathbb{Z}^{n},$$ where $0 \not \in \mathcal{B} \text{ is an auxiliary element},$ satisfying the following axioms:

\begin{enumerate}
    \item If $x,y \in \mathcal{B}$ then $e_i(x)=y \text{ if and only if } f_i(y)=x$. In this case, we further assume $\wt(y)-\wt(x)= \epsilon_i-\epsilon_{i+1},$ where $ \epsilon_i \in \mathbb{Z}^{n}$ whose $i^{th}$ entry is $1$ and others are $0$.
    \item $\varphi_i(x) - \varepsilon _i (x) = \wt(x) \cdot (\epsilon_i - \epsilon_{i+1})$ $\forall x \in \mathcal{B}$ and $i \in \{1, 2, \ldots , n-1\}$,
\end{enumerate}
where $\varepsilon_i,\varphi_i: \mathcal{B} \rightarrow \mathbb{Z}_+$ are defined as follows:
$$\varepsilon_i(x)=\max\{k|e_i^{k}x\neq 0\}
\text{  and  } \varphi_i(x)=\max\{k|f_i^{k}x\neq 0\}.$$
Our definition of crystals is called seminormal crystals in \cite[\S2.2]{Bump-Anne}. By a slight misuse of notation, a crystal is frequently denoted by its underlying set $\mathcal{B}$. The maps $e_i,f_i$ are called the raising and lowering operators respectively.
\begin{example}
\label{ex:standard crystal}
The \emph{standard} type $A_{n-1}$ crystal consists of the underlying set $[n]
   $ and the maps
\begin{equation*}
f_i(j) = \left\{
    \begin{array}{ll}
         i+1 & \quad \text{if} \quad j=i  \\
         0 & \quad \text{otherwise}
    \end{array}     \right. 
\text{ for $i \in [n-1]$} \quad  \text{ and } \quad \wt(i) = \epsilon _i \quad \text{ for $i \in [n]$}.    
\end{equation*}
We will denote this crystal by $\mathbb{W}_n$.
\end{example}
When $\mathcal{B}$ is a type $A_{n-1}$ crystal, we assign it a directed graph (which we call the \emph{crystal graph} of $ \mathcal{B}$) with vertices in $\mathcal{B}$ and edges labelled by $i \in [n-1]$. We draw an edge labelled by $i$ between two vertices $x,y$ begins from $x$ and terminates at $y$ if and only if $f_i(x)=y$. If the crystal graph, which is viewed as an undirected graph, is connected, we say that $\mathcal{B}$ is \emph{connected}. A subset $\mathcal{B}^{'}$ of a crystal $\mathcal{B},$ which is a union of connected components of $\mathcal{B},$ inherits a crystal structure from $\mathcal{B}$. In this case, we refer to $ \mathcal{B'}$ as a \emph{full subcrystal} of $\mathcal{B}$.

An element $b \in \mathcal{B}$ such that $e_i(b)=0$ for $1 \leq i \leq n-1$ is called a \textit{highest weight} element of $\mathcal{B}$. For example, the element $1$ is a (actually the only) highest weight element of the crystal $\mathbb{W}_n$.
\subsection{Tensor products of crystals} If $\mathcal{A} \text{ and } \mathcal{B}$ are two type $A_{n-1}$ crystals then the tensor product $\mathcal{A} \otimes \mathcal{B}$ is also a crystal of type $A_{n-1}$ whose underlying set is $\{x \otimes y: x \in \mathcal{A}, y \in \mathcal{B}\}$. $\wt(x \otimes y) \text{ is defined as }\wt(x) + \wt(y)$ and the raising, lowering operators are defined as follows:
\begin{equation*}\label{e action}
 e_i(x \otimes y) = \left\{ 
    \begin{array}{ll}
        e_i(x) \otimes y & \quad \text{if} \quad \varepsilon_i(x) > \phi_i(y)  \\
        x \otimes e_i(y) & \quad \text{if} \quad \varepsilon_i(x) \leq \phi_i(y) 
    \end{array}   \right. 
\end{equation*}
and 
\begin{equation*}\label{f action}
 f_i(x \otimes y) = \left\{ 
    \begin{array}{ll}
        f_i(x) \otimes y & \quad \text{if} \quad \varepsilon_i(x) \geq \phi_i(y)  \\
        x \otimes f_i(y) & \quad \text{if} \quad \varepsilon_i(x) < \phi_i(y) 
    \end{array}   \right. 
\end{equation*}
It is understood that $x \otimes 0 = 0 \otimes y = 0$. We adopt the convention for tensor products from \cite[\S2.3]{Bump-Anne}, which is opposite to convention given by Kashiwara \cite{Kashiwara:on-crystal-bases}.
\subsection{Crystal morphism}Let $\mathcal{A},\mathcal{B}$ be two type $A_{n-1}$ crystals. Then a \emph{crystal morphism} from $\mathcal{A}$ to $\mathcal{B}$ is a map $\psi:\mathcal{A} \rightarrow \mathcal{B} \sqcup \{0\}$ such that
\begin{enumerate}
\item If $a \in \mathcal{A}$ and $\psi(a) \in \mathcal{B},$ then $\wt(\psi(a))=\wt(a),$
$\varphi_i(\psi(a))=\varphi_i(a)$ for $1 \leq i \leq n-1$ and 
$\varepsilon_i(\psi(a))=\varepsilon_i(a)$ for $1 \leq i \leq n-1$;
\item $\psi(e_i a)=e_i\psi(a)$ provided $a\in \mathcal{A} ,e_ia \in \mathcal{A}$ and $\psi(a) \in \mathcal{B} ,\psi(e_ia) \in \mathcal{B}$;
\item $\psi(f_i a)=f_i\psi(a)$ provided $a\in \mathcal{A} ,f_ia \in \mathcal{A}$ and $\psi(a) \in \mathcal{B} ,\psi(f_ia) \in \mathcal{B}$.
\end{enumerate}
A morphism $\psi$ is said to be \emph{strict} if $\psi$ commutes with $e_i,f_i$ for $1 \leq i \leq n-1$. Moreover, a crystal morphism $\psi$ is called an \emph{embedding} or \emph{isomorphism} if the induced map $\psi : \mathcal{A} \sqcup \{0\} \rightarrow \mathcal{B} \sqcup \{0\}$ with $\psi(0)=0$ is a injection or bijection respectively.

The following propostion says that tensor products of crystals are associative.
\begin{proposition}\begin{upshape}
    \cite[Proposition 2.32]{Bump-Anne}
\end{upshape}
If $\mathcal{A},\mathcal{B},\mathcal{C}$ are crystals of type $A_{n-1}$ then the bijection $(\mathcal{A}\otimes\mathcal{B})\otimes\mathcal{C} \rightarrow \mathcal{A} \otimes (\mathcal{B} \otimes \mathcal{C})$ in which $(a \times b)\otimes c \mapsto a\otimes(b\otimes c)$ is a crystal isomorphism. We write $a \otimes b \otimes c$ to denote either $(a \times b)\otimes c$ or $(a \times b)\otimes c$.
\end{proposition}

Then the following lemma says how the crystal operators act on $k$-fold the tenosor product crystals.
\begin{lemma}
\begin{upshape}
    \cite[Lemma 2.33]{Bump-Anne}
\end{upshape}
\label{lemma:signature}
Let $\mathcal{B}_1,\mathcal{B}_2,\dots,\mathcal{B}_k$ be crystals of type $A_{n-1}$ and $x_i \in \mathcal{B}_{i}$ for $1 \leq i \leq k$. Then 
\begin{equation}
\label{eq:varphi_i}
    \varphi_i(x_1\otimes x_2 \otimes \cdots \otimes x_k)= \displaystyle\max_{j=1}^{k} \Biggl(\displaystyle\sum_{h=1}^{j}\varphi_i(x_h)- \displaystyle\sum_{h=1}^{j-1}\varepsilon_i(x_h)\Biggl)
\end{equation}
and if $j$ is the first value in \eqref{eq:varphi_i} where the maximum attained, then
\begin{equation}
\label{eq:f_i}
    f_i(x_1\otimes x_2 \otimes \cdots \otimes x_k)= x_1 \otimes \cdots \otimes f_i(x_j)\otimes \cdots \otimes x_k
\end{equation}
Similarly,
\begin{equation}
\label{eq:varepsilon_i}
    \varepsilon_i(x_k\otimes \cdots \otimes x_2 \otimes x_1)= \displaystyle\max_{j=1}^{k} \Biggl(\displaystyle\sum_{h=1}^{j}\varepsilon_i(x_h)- \displaystyle\sum_{h=1}^{j-1}\varphi_i(x_h)\Biggl)
\end{equation}
and if $j$ is the first value in \eqref{eq:varepsilon_i} where the maximum attained, then
\begin{equation}
\label{eq:e_i}
    e_i(x_k\otimes \cdots \otimes x_2 \otimes x_1)= x_k \otimes \cdots \otimes e_i(x_j)\otimes \cdots \otimes x_1
\end{equation}
\end{lemma}
The \emph{character} of a crystal $\mathcal{B}$ is defined by $ch(\mathcal{B}) := \displaystyle \sum _{ u \in \mathcal{B}} {\characx}^{\wt(u)}$. If $\mathcal{B}_1,\dots,\mathcal{B}_k$ are crystals then it is easy to see that $ ch(\mathcal{B}_{1} \otimes \mathcal{B}_{2} \otimes \cdots \otimes \mathcal{B}_{k})=ch(\mathcal{B}_{1})ch(\mathcal{B}_{2})\cdots ch(\mathcal{B}_{k})$.
\subsection{Crystal structure on set-valued tableaux of skew shape}
\label{def:crystal-SVT}
A crystal structure (of type $A_{n-1}$) on $\SVT_n(\lambda)$ is already given by Monical-Pechenik-Scrimshaw in \cite[\S3]{Travis:symmetric}. We extend it to a crystal structure on $\SVT_n(\lambda/\mu)$. Fix $T \in \SVT_n(\lambda/\mu)$ and $i \in [n-1]$. Write $-$ above each column of $T$ that contains an $i+1$ but not an $i$ and write $+$ above every column containing an $i$ but not $i+1$. Then by successively canceling $(-,+)$ pairs (in that ordered pair) we obtain a sequence as follows
\[
    \underbrace{+\cdots+}_{r}
    \underbrace{-\cdots-}_{s}
\]

$f_i(T):$ If $r=0$ then $f_i(T)=0$. Otherwise let $\mathbf{b}$ be the box that corresponds to the rightmost uncanceled $+$. Then $f_i(T)$ is given by any of the following:
\begin{itemize}
    \item if there is an box  $\mathbf{b}^{\rightarrow}$ immediately to the right of the box $\mathbf{b}$ that contains an $i$ (in this case, $i+1 \in \mathbf{b}^{\rightarrow}$) then $f_i(T)$ is given by removing $i$ from $\mathbf{b}^{\rightarrow}$ and adding an $i+1$ to $\mathbf{b}$.

    \item otherwise we change the $i$ in $\mathbf{b}$ to an $i+1$.
\end{itemize}

$e_i(T):$ If $s=0$ then $e_i(T)=0$. Otherwise let $\mathbf{b}$ be the box that corresponds to the leftmost uncanceled $-$.
Then $e_i(T)$ is given by any of the following:
\begin{itemize}
    \item if there is a box $\mathbf{b}^{\leftarrow}$ immediately to the left of the box $\mathbf{b}$ that contains an $i+1$ (in this case, $i \in \mathbf{b}^{\leftarrow}$) then $e_i(T)$ is given by removing the $i+1$ from $\mathbf{b}^{\leftarrow}$ and adding an $i$ to $\mathbf{b}$.
    \item otherwise we change the $i+1$ in $\mathbf{b}$ to an $i$.
\end{itemize}
The proof of well-defineness of the action of $e_i$ and $f_i$ is exactly same as in case of partition shape, see \cite[Lemma 3.2]{Travis:symmetric} for more details. It is clear that the axioms (1) and (2) in the definition of the crystals are also satisfied.

\bremark
It is clear from the definition that $ \ex(T)=\ex(e_i(T))=\ex(f_i(T))$ for any $T \in \SVT_n(\lambda/\mu),$ provided $ e_i(T),f_i(T) \in \SVT_n(\lambda/\mu)$ (see Figure~\ref{fig:crys-example}).
\eremark
\begin{figure}
    \centering
\begin{tikzpicture}[scale=0.8]
\draw (0,0)--(1.4,0)--(1.4,1.4);
\draw (0,0)--(0,0.7)--(1.4,0.7);
\draw (0.7,0)--(0.7,1.4)--(1.4,1.4);
\node at (0.35,0.35) {$1,2$};
\node at (1.05,0.35+0.025) {$2$};
\node at (1.05,1.05) {$1$};
\draw[red][->] (0.7,-0.2) -- (0.7,-1.2);
\node at (1,-0.7) {$\textcolor{red}{2}$};

\draw (0,-3)--(1.4,-3)--(1.4,-1.6);
\draw (0,-3)--(0,-2.3)--(1.4,-2.3);
\draw (0.7,-3)--(0.7,-1.6)--(1.4,-1.6);
\node at (0.35,-2.65) {$1,2$};
\node at (1.05,-2.65+0.025) {$3$};
\node at (1.05,-1.95) {$1$};
\draw[red][->] (0.2,-3.2) -- (-0.4,-4.2);
\node at (-0.3,-3.5) {$\textcolor{red}{2}$};
\draw[blue][->] (1.2,-3.2) -- (1.8,-4.2);
\node at (1.7,-3.5) {$\textcolor{blue}{1}$};

\draw (-1.5,-5.8)--(-0.1,-5.8)--(-0.1,-4.4);
\draw (-1.5,-5.8)--(-1.5,-5.1)--(-0.1,-5.1);
\draw (-0.8,-5.8)--(-0.8,-4.4)--(-0.1,-4.4);
\node at (0.35-1.5,0.35-5.8) {$1,3$};
\node at (1.05-1.5,0.35-5.8+0.025) {$3$};
\node at (1.05-1.5,1.05-5.8) {$1$};
\draw[blue][->] (-0.8,-6) -- (-0.8,-7);
\node at (-0.5,-6.5) {$\textcolor{blue}{1}$};

\draw (-1.5,-8.75)--(-0.1,-8.75)--(-0.1,-7.35);
\draw (-1.5,-8.75)--(-1.5,-8.05)--(-0.1,-8.05);
\draw (-0.8,-8.75)--(-0.8,-7.35)--(-0.1,-7.35);
\node at (0.35-1.5,0.35-8.75) {$1,3$};
\node at (1.05-1.5,0.35-8.75+0.025) {$3$};
\node at (1.05-1.5,1.05-8.75) {$2$};
\draw[blue][->] (-0.8,-8.95) -- (-0.8,-9.95);
\node at (-0.5,-9.5) {$\textcolor{blue}{1}$};

\draw (-1.5,-11.7)--(-0.1,-11.7)--(-0.1,-10.3);
\draw (-1.5,-11.7)--(-1.5,-11)--(-0.1,-11);
\draw (-0.8,-11.7)--(-0.8,-10.3)--(-0.1,-10.3);
\node at (0.35-1.5,0.35-11.7) {$2,3$};
\node at (1.05-1.5,0.35-11.7+0.025) {$3$};
\node at (1.05-1.5,1.05-11.7) {$2$};

\draw (1.5,-5.8)--(2.9,-5.8)--(2.9,-4.4);
\draw (1.5,-5.8)--(1.5,-5.1)--(2.9,-5.1);
\draw (2.2,-5.8)--(2.2,-4.4)--(2.9,-4.4);
\node at (0.35+1.5,0.35-5.8) {$1,2$};
\node at (1.05+1.5,0.35-5.8+0.025) {$3$};
\node at (1.05+1.5,1.05-5.8) {$2$};
\end{tikzpicture}
\begin{tikzpicture}
[scale=0.8]
\node at (0,0) {$\null$};
\node at (1,0) {$\null$};
\end{tikzpicture}
\begin{tikzpicture}
[scale=0.8]
\draw (0,0)--(1.4,0)--(1.4,1.4);
\draw (0,0)--(0,0.7)--(1.4,0.7);
\draw (0.7,0)--(0.7,1.4)--(1.4,1.4);
\node at (0.35,0.35+0.025) {$1$};
\node at (1.05,0.35) {$2,3$};
\node at (1.05,1.05) {$1$};
\draw[blue][->] (0.7,-0.2) -- (0.7,-1.2);
\node at (1,-0.7) {$\textcolor{blue}{1}$};

\draw (0,-2.95)--(1.4,-2.95)--(1.4,-1.55);
\draw (0,-2.95)--(0,-2.25)--(1.4,-2.25);
\draw (0.7,-2.95)--(0.7,-1.55)--(1.4,-1.55);
\node at (0.35,0.35-2.95+0.025) {$2$};
\node at (1.05,0.35-2.95) {$2,3$};
\node at (1.05,1.05-2.95) {$1$};
\draw[red][->] (0.7,-3.15) -- (0.7,-4.15);
\node at (1,-3.7) {$\textcolor{red}{2}$};

\draw (0,-5.9)--(1.4,-5.9)--(1.4,-4.5);
\draw (0,-5.9)--(0,-5.2)--(1.4,-5.2);
\draw (0.7,-5.9)--(0.7,-4.5)--(1.4,-4.5);
\node at (0.35,0.35-5.9) {$2,3$};
\node at (1.05,0.35-5.9+0.025) {$3$};
\node at (1.05,1.05-5.9) {$1$};
\node at (0.35-1.5,0.35-11.7) {$\null$};
\end{tikzpicture}
\caption{Crystal structure on the set of all elements in $\SVT_{3}((2,2)/(1)),$ whose excess is $(0,1)$.}
\label{fig:crys-example}
\end{figure}
\subsection{Demazure crystals}For $w \in S_n, \lambda \in \mathcal{P}[n],$ the \emph{Demazure crystal} $\mathcal{B}_{w} (\lambda)$ is defined as:
\begin{equation*}
\label{eq:demcrys}
  \mathcal{B}_{w} (\lambda) := \{f_{i_1} ^{k_1} f_{i_2} ^{k_2} \cdots f_{i_p} ^{k_p}  T_{\lambda} : k_j \geq0 \} \textcolor{black}{\setminus \{0\}}, 
\end{equation*}
where $s_{i_1} s_{i_2} \cdots s_{i_p}$ is any reduced expression of $w$ and  $T_{\lambda}$ is the unique semi-standard Young tableau of shape and weight both equal to $\lambda$. 

Clearly, $\mathcal{B}_{w} (\lambda)$ is a certain subset of $ \Tab_n(\lambda)$. For instance, see Figure~\ref{fig:Demazure}. It is well-known that the Demazure crystal $\mathcal{B}_{w} (\lambda)$ is independent of the chosen reduced expression of $w$ \cite[Theorem 13.5]{Bump-Anne}. The following proposition is the refined Demazure character formula in \cite{Kashiwara:refine-Demazure}.
\begin{proposition}    
For $\lambda \in \mathcal{P}[n]$ and $ w \in S_n,$ $\displaystyle\sum_{T \in \mathcal{B}_{w}(\lambda)}\characx^{\wt(T)} = \key_{w.\lambda}$.
\end{proposition}
\begin{figure}
    \centering
\begin{tikzpicture}[scale=0.8]
    \draw (0,0)--(0,1.4)--(1.4,1.4);
    \draw (0,0)--(0.7,0)--(0.7,1.4);
    \draw (0,0.7)--(1.4,0.7)--(1.4,1.4);
    \node at (0.35,0.35) {$2$};
    \node at (0.35,1.05) {$1$};
    \node at (1.05,1.05) {$1$};
    \draw[red][->] (0.1,-0.2) -- (-0.55,-1.2);
    \node at (-0.45,-0.5) {$\textcolor{red}{2}$};
    \draw[blue][->] (0.85,-0.2) -- (1.5,-1.2);
    \node at (1.45,-0.5) {$\textcolor{blue}{1}$};

    \draw (0-1.5,0-2.8)--(0-1.5,1.4-2.8)--(1.4-1.5,1.4-2.8);
    \draw (0-1.5,0-2.8)--(0.7-1.5,0-2.8)--(0.7-1.5,1.4-2.8);
    \draw (0-1.5,0.7-2.8)--(1.4-1.5,0.7-2.8)--(1.4-1.5,1.4-2.8);
    \node at (0.35-1.5,0.35-2.8) {$3$};
    \node at (0.35-1.5,1.05-2.8) {$1$};
    \node at (1.05-1.5,1.05-2.8) {$1$};
    \draw[blue][->] (0.35-1.5,-0.2-2.8) -- (0.35-1.5,-1.2-2.8);
    \node at (0.6-1.5,-0.5-2.8) {$\textcolor{blue}{1}$};

    \draw (0-1.5,0-2.8-2.8)--(0-1.5,1.4-2.8-2.8)--(1.4-1.5,1.4-2.8-2.8);
    \draw (0-1.5,0-2.8-2.8)--(0.7-1.5,0-2.8-2.8)--(0.7-1.5,1.4-2.8-2.8);
    \draw (0-1.5,0.7-2.8-2.8)--(1.4-1.5,0.7-2.8-2.8)--(1.4-1.5,1.4-2.8-2.8);
    \node at (0.35-1.5,0.35-2.8-2.8) {$3$};
    \node at (0.35-1.5,1.05-2.8-2.8) {$1$};
    \node at (1.05-1.5,1.05-2.8-2.8) {$2$};
    \draw[blue][->] (0.35-1.5,-0.2-2.8-2.8) -- (0.35-1.5,-1.2-2.8-2.8);
    \node at (0.6-1.5,-0.5-2.8-2.8) {$\textcolor{blue}{1}$};

    \draw (0-1.5,0-2.8-2.8-2.8)--(0-1.5,1.4-2.8-2.8-2.8)--(1.4-1.5,1.4-2.8-2.8-2.8);
    \draw (0-1.5,0-2.8-2.8-2.8)--(0.7-1.5,0-2.8-2.8-2.8)--(0.7-1.5,1.4-2.8-2.8-2.8);
    \draw (0-1.5,0.7-2.8-2.8-2.8)--(1.4-1.5,0.7-2.8-2.8-2.8)--(1.4-1.5,1.4-2.8-2.8-2.8);
    \node at (0.35-1.5,0.35-2.8-2.8-2.8) {$3$};
    \node at (0.35-1.5,1.05-2.8-2.8-2.8) {$2$};
    \node at (1.05-1.5,1.05-2.8-2.8-2.8) {$2$};

    \draw (0-1.5+2.5,0-2.8)--(0-1.5+2.5,1.4-2.8)--(1.4-1.5+2.5,1.4-2.8);
    \draw (0-1.5+2.5,0-2.8)--(0.7-1.5+2.5,0-2.8)--(0.7-1.5+2.5,1.4-2.8);
    \draw (0-1.5+2.5,0.7-2.8)--(1.4-1.5+2.5,0.7-2.8)--(1.4-1.5+2.5,1.4-2.8);
    \node at (0.35-1.5+2.5,0.35-2.8) {$2$};
    \node at (0.35-1.5+2.5,1.05-2.8) {$1$};
    \node at (1.05-1.5+2.5,1.05-2.8) {$2$};
\end{tikzpicture}
\caption{The Demazure crystal $\mathcal{B}_{s_1s_2} (2,1,0) $}
\label{fig:Demazure}
\end{figure}
\section{proof of the main theorem}
\label{Section 4}
In this section, given a skew shape $\lambda/\mu$ $(\lambda, \mu \in \mathcal{P}[n])$ and a flag $\Phi \in \mathcal{F}[n],$ we provide a Demazure crystal structure on $\SVT(\lambda/\mu, \Phi)$. 
\begin{definition}
The \emph{row reading word} of a semi-standard Young tableau $T$, denoted by $r_T$, is obtained by reading the entries of $T$ row-by-row, starting from bottom row, left to right and continuing up the rows.
\end{definition}
\begin{example}
$T= \ytableausetup{mathmode,
notabloids}
  \begin{ytableau}
  \none &1&2\\2&4 \\3  
  \end{ytableau}
$ is a semi-standard Young tableau of shape $(3,2,1)/(1)$ with row reading word $r_{T}=32412$.
\end{example}
For $\alpha \in \mathbb{Z}^n_{+} ,$ we denote by $\mathbf{b}(\alpha)$ the word $ b^{(n)} \cdots b^{(2)}b^{(1)}$ in which $b^{(j)}$ consists of a string of $\alpha_j$ copies of $j$. For instance, $\mathbf{b}(2,3,0,1) = 422211$.

Given $\alpha \in \mathbb{Z}^n_{+} $ and a flag $\Phi \in \mathcal{F}[n],$ we define $\mathcal{W}(\alpha, \Phi)$ as the set of all words $\mathbf{v}= v^{(n)} \cdots v^{(2)} v^{(1)}$ in $\{1,2,\dots\}$ such that each $v^{(i)}$ is a maximal row word (i.e., the last letter of $v^{(i)}$ is greater than the first letter of $v^{(i-1)}$) of length $\alpha_i$ together with the following properties:
\begin{itemize}
    \item each letter in $v^{(i)}$ can be at most $\Phi_i$.
    \item $(
    \begin{bmatrix}
    \mathbf{b}(\alpha)\\
    \mathbf{v}
    \end{bmatrix} \rightarrow \emptyset) =(-$, $\text{key}(\alpha)) ,$ where $\text{key}(\alpha)$ is the unique semi-standard Young tableau of shape $\alpha^{\dagger}$ and weight $\alpha$.
\end{itemize}
\begin{example}
    Let $\alpha=(1,2,0,1)$ and $ \Phi=(1,2,3,4)$. Then the set of all words $\mathbf{v}= v^{(4)} v^{(3)} v^{(2)} v^{(1)}$ in $\mathcal{W}(\alpha, \Phi)$ are given below:
    $$ 3\cdot\cdot 12 \cdot 1 \quad 3\cdot \cdot 22 \cdot 1 \quad 4\cdot \cdot 22 \cdot1 $$
\end{example}
\begin{theorem}
\begin{upshape} \cite[Theorem 21]{RS} \end{upshape}
\label{theorem:RS-hat}
For $\beta \in \mathbb{Z}_+^{n}$ and flag $\Phi \in \mathcal{F}[n]$, either $\mathcal{W}(\beta, \Phi)$ is empty or there is a bijection $\zeta$ between the sets $\mathcal{W}(\beta, \Phi)$ and $ \mathcal{W}(\widehat{\beta}, \Phi_0)$ for some $\widehat{\beta} \in \mathbb{Z}_+^{n}$ with $\beta^{\dagger} = \widehat{\beta}^{\dagger}$ such that if $\mathbf{u} \mapsto \zeta (\mathbf{u})$ then $\mathbf{u}$ and $\zeta (\mathbf{u})$ are Knuth equivalent. Here $\Phi_0$ denotes the standard flag $(1,2,\dots,n)$.
\end{theorem}
\begin{theorem}
\begin{upshape} \cite[Proposition 5.6]{LS-keys} \end{upshape}
Let $\alpha = w.\alpha ^{\dagger}$. Then the set $\mathcal{W}(\alpha, \Phi_0)$ has a one-to-one correspondence with the set $\mathcal{B}_{w}(\alpha^{\dagger})$ via $\mathbf{u} \mapsto P(\mathbf{u})$ where $ P(\mathbf{u})$ is the unique tableau that is Knuth equivalent to $\mathbf{u}$.
\end{theorem}
For a skew shape $\lambda/\mu,$ we write $\mathbf{b}(\lambda/\mu)$ to denote the word $\mathbf{b}(\lambda-\mu)$.
Let $\Tab(\lambda/\mu, \Phi)$ be the set of all semi-standard Young tableaux in $\SVT(\lambda/\mu, \Phi)$.

A word $y = y_1 y_2 \cdots y_s$ is called a \textit{Yamanouchi word}~\cite[\S 5.2]{Fulton:yt} if, for every $t \geq 1$, the number of $i$’s appearing in $y_t \cdots y_s$ is at least the number of $(i+1)$’s appearing there, for all $i \geq 1$. For example, $3231211$ is a Yamanouchi word, whereas $3112$ is not.

We say a semi-standard Young tableau $R$ is $(\lambda/\mu, \Phi)$-compatible if there exists a unique tableau $T_0 \in \Tab(\lambda/\mu,\Phi)$ such that $r_{T_0}$ is Yamanouchi word along with the biword 
$ \begin{bmatrix}
           \mathbf{b}(\lambda/\mu) \\
           r_{T_0} \\
\end{bmatrix} $ 
corresponds to $ (T_{\sh(R)}, R)$ under the Burge correspondence, Theorem~\ref{Theorem:Burge}. Here $T_{\sh(R)}$ is the unique semi-standard Young tableau of shape and weight both equal to $\sh(R)$.

Given a $(\lambda/\mu, \Phi)$-compatible tableau $R,$ we define the following \cite[Appendix]{KRSV}
$$\mathcal{A}(R, \lambda/\mu,\Phi):=\{ T \in \Tab(\lambda/\mu,\Phi) :  
(\begin{bmatrix}
           \mathbf{b}(\lambda/\mu) \\
           r_T \\
\end{bmatrix} \rightarrow \emptyset) = (\rect (T), R)\},$$
where $\rect(T)$ denotes the unique semi-standard Young tableau Knuth equivalent to $r_{T}$.

Thus $\Tab(\lambda/\mu,\Phi)=\displaystyle\bigsqcup_{R}\mathcal{A}(R, \lambda/\mu,\Phi),$ where $R$ ranges over all $(\lambda/\mu, \Phi)$-compatible tableaux. We proved the following propositions in \cite{KRSV} which produces a Demazure crystal structure on $\Tab(\lambda/\mu, \Phi)$.
\begin{proposition}
\begin{upshape}
    \cite[Proposition A.7]{KRSV}
\end{upshape}
A bijection $\Omega$ exists between the sets $\mathcal{A}(R, \lambda/\mu, \Phi)$ and $ \mathcal{W}(\beta(R), \Phi)$ such that, if $ T \mapsto \Omega (T),$ then $r_T$ and $\Omega{(T)}$ are Knuth equivalent. Here $\beta(R)$ denotes the weight of the left key tableau $K_{\_}(R)$ of $R$.   
\end{proposition}
\begin{proposition}
\begin{upshape}
    \cite[Proposition A.9]{KRSV}
\end{upshape}
\label{proposition:tableau}
The rectification map $ \rect: \mathcal{A}(R, \lambda/\mu,\Phi) \rightarrow  \mathcal{B}_{w}(\widehat{\beta(R)}^{\dagger})$ is a weight-preserving bijection which intertwines the crystal raising and lowering operators. Here $w$ is any permutation such that $w. \widehat{\beta(R)}^{\dagger}=\widehat{\beta(R)}$.   
\end{proposition} 
If $\theta, \theta'$ are two skew shapes, we let $\theta * \theta'$ be the skew shape obtained by putting $\theta$ and $\theta'$ corner to corner as shown below:
$$
\begin{tikzpicture}[scale=0.9]
    \draw (0,0) -- (0,0.7);
    \draw (0.7,1.4)--(2.1,1.4);
    \draw (0,0)-- (0.7,0) -- (0.7,1.4);
    \draw (0.7,0)-- (1.4,0) -- (1.4,1.4);
    \draw (0,0.7) -- (2.1,0.7) -- (2.1,1.4);
    \node at (-0.6,0.7) {$\theta=$};
    \node at (0,-0.35) {$\null$};
\end{tikzpicture}
\begin{tikzpicture}
    \node at (0,0) {$\null$};
    \node at (0.5,0) {$\null$}; 
\end{tikzpicture}
\begin{tikzpicture}[scale=0.9]
  \draw (0,0)--(1.4,0)--(1.4,1.4);
  \draw (0,0)--(0,0.7)--(1.4,0.7);
  \draw (0.7,0)--(0.7,1.4)--(1.4,1.4);
  \node at (-0.6,0.7) {$\theta'=$};
  \node at (0,-0.35) {$\null$};
\end{tikzpicture}
\begin{tikzpicture}
    \node at (0,0) {$\null$};
    \node at (0.5,0.5) {$\implies$};
    \node at (0,-0.35) {$\null$};
\end{tikzpicture}
\begin{tikzpicture}
    \draw (0,0) -- (0,0.7);
    \draw (0.7,1.4) -- (2.1,1.4);
    \draw (0,0) -- (0.7,0) -- (0.7,1.4);
    \draw (0.7,0) -- (1.4,0) -- (1.4,1.4);
    \draw (0,0.7) -- (2.1,0.7) -- (2.1,1.4);
    \node at (-1,0.7+0.35) {$\theta * \theta' =$};
    \draw (0+2.1,0+1.4)--(1.4+2.1,0+1.4)--(1.4+2.1,1.4+1.4);
    \draw (0+2.1,0+1.4)--(0+2.1,0.7+1.4)--(1.4+2.1,0.7+1.4);
    \draw (0.7+2.1,0+1.4)--(0.7+2.1,1.4+1.4)--(1.4+2.1,1.4+1.4);
\end{tikzpicture}
$$
If $T,T'$ are semi-standard Young tableaux of shapes $\theta, \theta'$ respectively then we define another semi-standard Young tableau $T*T'$ of skew shape $\theta*\theta'$ as follows:
$$
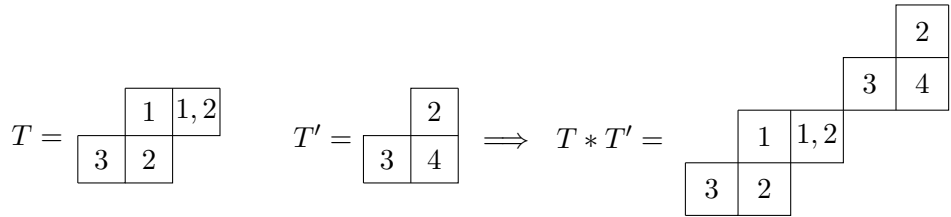
\begin{tikzpicture}[scale=0.9]
    \draw (0,0) -- (0,0.7);
    \draw (0.7,1.4)--(2.1,1.4);
    \draw (0,0)-- (0.7,0) -- (0.7,1.4);
    \draw (0.7,0)-- (1.4,0) -- (1.4,1.4);
    \draw (0,0.7) -- (2.1,0.7) -- (2.1,1.4);
    \node at (0.35,0.35) {$2$};
    \node at (1.05,0.35) {$3$};
    \node at (1.05,1.05) {$1$};
    \node at (1.75,1.05) {$1$};
    \node at (-0.6,0.7) {$T=$};
    \node at (0,-0.35) {$\null$};
\end{tikzpicture}
\begin{tikzpicture}
    \node at (0,0) {$\null$};
    \node at (0.5,0) {$\null$}; 
\end{tikzpicture}
\begin{tikzpicture}[scale=0.9]
  \draw (0,0)--(1.4,0)--(1.4,1.4);
  \draw (0,0)--(0,0.7)--(1.4,0.7);
  \draw (0.7,0)--(0.7,1.4)--(1.4,1.4);
  \node at (0.35,0.35) {$3$};
  \node at (1.05,0.35) {$4$};
  \node at (1.05,1.05) {$2$};
  \node at (-0.6,0.7) {$T'=$};
  \node at (0,-0.35) {$\null$};
\end{tikzpicture}
\begin{tikzpicture}
    \node at (0,0) {$\null$};
    \node at (0.5,0.5) {$\implies$};
    \node at (0,-0.35) {$\null$};
\end{tikzpicture}
\begin{tikzpicture}
\draw (0,0) -- (0,0.7);
    \draw (0.7,1.4)--(2.1,1.4);
    \draw (0,0)-- (0.7,0) -- (0.7,1.4);
    \draw (0.7,0)-- (1.4,0) -- (1.4,1.4);
    \draw (0,0.7) -- (2.1,0.7) -- (2.1,1.4);
    \node at (0.35,0.35) {$2$};
    \node at (1.05,0.35) {$3$};
    \node at (1.05,1.05) {$1$};
    \node at (1.75,1.05) {$1$};
    \node at (-1,0.7+0.35) {$T * T'=$};
    \draw (0+2.1,0+1.4)--(1.4+2.1,0+1.4)--(1.4+2.1,1.4+1.4);
    \draw (0+2.1,0+1.4)--(0+2.1,0.7+1.4)--(1.4+2.1,0.7+1.4);
    \draw (0.7+2.1,0+1.4)--(0.7+2.1,1.4+1.4)--(1.4+2.1,1.4+1.4);
    \node at (0.35+2.1,0.35+1.4) {$3$};
    \node at (1.05+2.1,0.35+1.4) {$4$};
    \node at (1.05+2.1,1.05+1.4) {$2$};
    \node at (3.5,2.8+0.35) {$\null$};
\end{tikzpicture}
$$
Given skew shapes $ \theta_i$ for $i=1,2,3,$ we denote $\theta_3 * (\theta_2 * \theta_1)$ as $\theta_3 * \theta_2 *\theta_1$. Also, if $T_i$ are semi-standard Young tableaux of shape $\theta_i $ for $i=1,2,3$ then we denote $T_3*(T_2*T_1)$ as $T_3*T_2*T_1$. 

\begin{definition}
The \emph{reading word} $w(S)$ of a set-valued tableau $S$ is the word obtained by reading each row of $S$, starting from the bottom row, according to the following procedure, and then continuing up the rows. In each row, we first ignore the smallest entry of each box, and read the remaining entries from right to left and from largest to smallest within each cell. Then we read the smallest entry of each cell from left to right. 
\end{definition}
\begin{example}
$S=$ $  \ytableausetup{mathmode,
notabloids}
  \begin{ytableau}
    1&1,2&2,3\\2,3&4   
  \end{ytableau}
$ is a set-valued tableau of shape $(3,2,0)$ with $w(S)=32432112$.
\end{example}

For a skew shape $\lambda/\mu$ $(\lambda,\mu \in \mathcal{P}[n])$ and $\mathbf{e} \in \mathbb{Z}_{+}^{n},$ we define a new skew shape as below
$$\sigma^{\mathbf{e}}_{\lambda/\mu}:=(\lambda_n-\mu_n,1^{e_n}) *\cdots * (\lambda_2-\mu_2,1^{e_2}) * (\lambda_1-\mu_1,1^{e_1}). $$
Then each semi-standard set-valued tableau $T$ of shape $\lambda/\mu$ with excess $\mathbf{e}$ corresponds to a unique semi-standard Young tableau $\Tilde{T} =\Tilde{T_n} * \cdots *\Tilde{T_2} *\Tilde{T_1}$ of shape $\sigma^{\mathbf{e}}_{\lambda/\mu}$, where each $\Tilde{T_i}$ is a semi-standard Young tableau such that $r_{\Tilde{T_i}}$ is the row reading word of $i^{th}$ row $T_i$ of $T$ and $\sh(\Tilde{T_i})=(\lambda_i-\mu_i, 1^{e_i})$ for $1 \leq i \leq n$.
\begin{example}
    Let us assume
$$
\begin{tikzpicture}[scale=1.2]
  \draw (0,0)--(1.6,0)--(1.6,1.6)--(2.4,1.6);
  \draw (0,0)--(0,0.8)--(2.4,0.8)--(2.4,1.6);
  \draw (0.8,0)--(0.8,1.6)--(1.6,1.6);
  \node at (0.4,0.4) {$2,\textcolor{blue}{3}$};
  \node at (1.2,0.4) {$3,\textcolor{blue}{4}$};
  \node at (1.2,1.2) {$1,\textcolor{blue}{2}$};
  \node at (2,1.2) {$2,\textcolor{blue}{3},\textcolor{blue}{4}$};
  \node at (-0.4,0.8) {$T=$};
\end{tikzpicture}
$$
Then we have
$$
\begin{tikzpicture}[scale=1.2]
    \draw (0,0)--(0,0.8)--(1.6,0.8)--(1.6,0)--(0,0);
    \draw (0.8,0)--(0.8,0.8);
    \node at (0.4,0.4) {$1,\textcolor{blue}{2}$}; 
    \node at (1.2,0.4) {$2,\textcolor{blue}{3},\textcolor{blue}{4}$};
    \node at (-0.4,0.4) {$T_1=$};
\end{tikzpicture}
\begin{tikzpicture}
    \node at (0,0) {$\null$};
    \node at (0.5,0) {$\null$}; 
\end{tikzpicture}
\begin{tikzpicture}[scale=1.2]
    \draw (0,0)--(0,0.8)--(1.6,0.8)--(1.6,0)--(0,0);
    \draw (0.8,0)--(0.8,0.8);
    \node at (0.4,0.4) {$2,\textcolor{blue}{3}$}; 
    \node at (1.2,0.4) {$3,\textcolor{blue}{4}$};
    \node at (-0.4,0.4) {$T_2=$};
\end{tikzpicture}
$$
Therefore,
$$ \ytableausetup{mathmode,
notabloids}
\Tilde{T_1}= \begin{ytableau}
     1 &2\\ \textcolor{blue}{2} \\ \textcolor{blue}{3} \\ \textcolor{blue}{4}  
\end{ytableau}\text{\hspace{0.5 cm} and \hspace{0.5 cm}}
\ytableausetup{mathmode,
notabloids}
\Tilde{T_2}=\begin{ytableau}
      2 & 3\\ \textcolor{blue}{3} \\\textcolor{blue}{4}
\end{ytableau} \implies 
\ytableausetup{mathmode,
notabloids}
\Tilde{T_2}*\Tilde{T_1} = \begin{ytableau}
     \none & \none & 1 &2\\
     \none & \none & \textcolor{blue}{2} \\ \none & \none & \textcolor{blue}{3} \\ 
     \none & \none & \textcolor{blue}{4} \\  
      2 & 3\\ \textcolor{blue}{3} \\\textcolor{blue}{4}
\end{ytableau}$$
\end{example}
For $\alpha=(\alpha_1,\dots,\alpha_m) \in \mathbb{Z}^m _{+},\beta=(\beta_1,\dots,\beta_n) \in \mathbb{Z}^n _{+},$ we write $ \alpha * \beta =(\alpha_1,\dots,\alpha_m,\beta_1,\dots,\beta_n)$. For positive integers $s,t,$ we define $s^t:=(s,s,\dots,s) \in \mathbb{Z}^t_{+}$. Let $ \SVT_{\mathbf{e}}(\lambda/\mu, \Phi)$ be the set of all set-valued tableaux in $\SVT(\lambda/\mu, \Phi)$ with excess $\mathbf{e}$. Then if $T \in \SVT_{\mathbf{e}}(\lambda/\mu, \Phi)$ and $\Tilde{T}=\Tilde{T_n}*\cdots * \Tilde{T_2} * \Tilde{T_1}$ then $\Tilde{T_i}$ respects the flag $\phi_i^{e_i+1}$ for $1 \leq i \leq n$. Thus $\Tilde{T}\in \Tab(\sigma^{\mathbf{e}}_{\lambda/\mu},\Phi^{\mathbf{e}}),$ where $\Phi^{\mathbf{e}}:=\phi_1^{e_1+1}*\phi_2^{e_2+1}*\cdots * \phi_n^{e_n+1}$.

Let $P$ be a highest weight element in $\SVT(\lambda/\mu, \Phi)$ and $\ex(P)=\mathbf{e}$. Then $\Tilde{P}=\Tilde{P_n} * \cdots * \Tilde{P_2} *\Tilde{P_1}$ and $\sh(\Tilde{P})=\sigma^{\mathbf{e}}_{\lambda/\mu}$. Now if $(
\begin{bmatrix}
      \mathbf{b}(\sigma^{\mathbf{e}}_{\lambda/\mu}) \\
       r_{\Tilde{P}} \\
\end{bmatrix}
\rightarrow \emptyset)=(\rect (\Tilde{P}), \hat{P})$ (Theorem~\ref{Theorem:Burge}), we say $\hat{P}$ is a $(\lambda/\mu,\Phi)$-compatible tableau for SVT. We also define $\overline{\ex}(\hat{P}):=\ex(P)=\mathbf{e}$.
\begin{example}
    Let $\lambda = (4,2,2), \mu =(2,1,0), \mathbf{e}=(1,1,0), \Phi=(2,2,3)$. Also, we take
$$ P= \ytableausetup{mathmode,
notabloids}
\begin{ytableau}
    \none & \none &1 &1,2\\ \none & 1,2 \\ 1 &3  
\end{ytableau} \in \SVT_{\mathbf{e}}(\lambda/\mu, \Phi)
$$
It is easy to see that $P$ is a highest weight element of $\SVT_{\mathbf{e}}(\lambda/\mu, \Phi)$.
Also, $\Tilde{P}=\Tilde{P_3} * \Tilde{P_2} * \Tilde{P_1}$, where 
$$
\Tilde{P_1}= \ytableausetup{mathmode,
notabloids}
\begin{ytableau}
    1 &1\\2  
\end{ytableau} \hspace{1 cm} 
\Tilde{P_2}= \ytableausetup{mathmode,
notabloids}
\begin{ytableau}
    1 \\ 2  
\end{ytableau} \hspace{1 cm}
\Tilde{P_3}= \ytableausetup{mathmode,
notabloids}
\begin{ytableau}
    1 & 3  
\end{ytableau}$$
Now, $$(
\begin{bmatrix}
           \mathbf{b}(\sigma^{\mathbf{e}}_{\lambda/\mu}) \\
           r_{\Tilde{P}} \\
           
\end{bmatrix} \rightarrow \emptyset)
=(\ytableausetup{mathmode,
notabloids}
  \begin{ytableau}
    1&1&1&1\\2 &2 \\3  
  \end{ytableau},\ytableausetup{mathmode,
notabloids}
  \begin{ytableau}
    1&1&3&5\\2 &4\\5
  \end{ytableau} ).$$
So $\hat{P} = \ytableausetup{mathmode,
notabloids}
\begin{ytableau}
    1&1&3&5\\2 &4\\5
  \end{ytableau} $ is a $(\lambda/\mu, \Phi)$-compatible tableau for SVT such that $\overline{\ex}(\hat{P})=\ex(P)=(1,1,0)$.
\end{example}

Let $P$ be any highest weight element of $\SVT_{\mathbf{e}}(\lambda/\mu, \Phi)$ and $ \SVT_{\mathbf{e}}(\lambda/\mu, \Phi;P)$ be the connected component of the crystal graph of $\SVT_{\mathbf{e}}(\lambda/\mu, \Phi)$ containing $P$. Then 
\begin{equation}
\label{eq:Demazure:excess}
\SVT_{\mathbf{e}}(\lambda/\mu, \Phi) = \bigsqcup_{P} \SVT_{\mathbf{e}}(\lambda/\mu, \Phi;P),    
\end{equation}
where $P$ varies over the set of all highest weight elements in $\SVT_{\mathbf{e}}(\lambda/\mu, \Phi)$.

Now we show that $\SVT_{\mathbf{e}}(\lambda/\mu, \Phi;P)$ is isomorphic to a Demazure crystal, which implies $\SVT_{\mathbf{e}}(\lambda/\mu, \Phi)$ admits a Demazure crystal structure. In order to do that we define the following map 
$$\Psi : \SVT_{\mathbf{e}}(\lambda/\mu, \Phi;P) \rightarrow \mathcal{A}(\hat{P}, \sigma^{\mathbf{e}}_{\lambda/\mu},\Phi^{\mathbf{e}}) \text{ by }
S \mapsto \Tilde{S} =\Tilde{S_n} * \cdots * \Tilde{S_2}*\Tilde{S_1}.$$
Then the following proposition and Proposition~\ref{proposition:tableau} provides a Demazure crystal structure on $\SVT_{\mathbf{e}}(\theta, \Phi)$.

\begin{proposition}
\label{Proposition:main}
The map $\Psi : \SVT_{\mathbf{e}}(\lambda/\mu, \Phi;P)  \rightarrow \mathcal{A}(\hat{P}, \sigma^{\mathbf{e}}_{\lambda/\mu},\Phi^{\mathbf{e}})$ defined by $ \Psi(S)=\Tilde{S}$ is a weight-preserving bijection which intertwines the crystal raising and lowering operators.
\end{proposition}
\begin{proof}
First we check the map $\Psi$ is well-defined. Let $S \in \SVT_{\mathbf{e}}(\lambda/\mu, \Phi;P)$ and $S_i$ be the $i^{th}$ row of $S$. Also, let $\Tilde{S}=\Tilde{S_n}* \cdots * \Tilde{S_2}*\Tilde{S_1}$. Then using \cite[Remark 3.7]{Travis:symmetric},
$\varphi_i(S_j)= \varphi_i(\Tilde{S_j}), \varepsilon_i(S_j)=\varepsilon_i(\Tilde{S_j}) $ for $1 \leq i \leq n-1$ and $1 \leq j \leq n$. Then considering $S$ as $ S_n \otimes \cdots \otimes S_2  \otimes S_1$ and using \cite[Remark 3.7]{Travis:symmetric},
Lemma~\ref{lemma:signature}, we can say that $\Psi$ commutes with $ e_i,f_i$. Now $S=f_{i_1}^{k_1}f_{i_2}^{k_2}\cdots f_{i_t}^{k_t}(P)$. Thus, using \cite[Proposition 29]{RS}, $\Psi(S)=\Tilde{S}=f_{i_1}^{k_1}f_{i_2}^{k_2}\cdots f_{i_t}^{k_t}(\Tilde{P}) \in \mathcal{A}(\hat{P}, \sigma^{\mathbf{e}}_{\lambda/\mu},\Phi^{\mathbf{e}})$. Hence the map $\Psi$ is well-defined. It is clear that $\Psi$ is weight preserving, i.e., $\wt(S)=\wt(\Psi(S))$, for any $ S \in\SVT_{\mathbf{e}}(\lambda/\mu, \Phi;P)$.

Let $\Psi(S)=\Psi(T) \implies \Tilde{S_n} * \cdots * \Tilde{S_2} * \Tilde{S_1}= \Tilde{T_n} * \cdots * \Tilde{T_2} * \Tilde{T_1} \implies \Tilde{S_i}=\Tilde{T_i}$ for all $i \in [n]$. Hence $S_i=T_i$ for $1 \leq i \leq n$. Therefore $S=T$. So $\Psi$ is injective. Let $\Tilde{T} \in \mathcal{A}(\hat{P}, \sigma^{\mathbf{e}}_{\lambda/\mu},\Phi^{\mathbf{e}})$. Then $\Tilde{T}=f_{j_1}^{l_1}f_{j_2}^{l_2}\cdots f_{j_t}^{l_t}(\Tilde{P})$. Let $\SVT^{\mathbf{e}}_n (\lambda/\mu)$ denote the set of all set-valued tableaux with excess $\mathbf{e}$ in $\SVT_n(\lambda/\mu).$
Using \cite[Proposition 3.8]{Travis:symmetric}, we can say that the crystal structure on $\SVT^{\mathbf{e}}_n (\lambda/\mu) $ given in \ref{def:crystal-SVT} can be seen by the following embedding
$$ \chi^{\mathbf{e}}_{\lambda/\mu}: \SVT^{\mathbf{e}}_n (\lambda/\mu) \rightarrow \mathbb{W}_n^{\otimes|\lambda|-|\mu| + |\mathbf{e}|}, T \mapsto w(T)=v_1v_2\cdots v_k \mapsto v_1 \otimes v_2 \otimes \cdots \otimes v_k.$$ Similarly, the crystal structure on $\Tab_n(\theta)=\SVT^{\mathbf{0}}_n(\theta)$ can be defined by the embedding map $\chi^{\mathbf{0}}_{\theta},$ where $\theta= \sigma^{\mathbf{e}}_{\lambda/\mu}$. Also, $w(R)=r_{\Tilde{R}}$ for any $R \in \SVT^{\mathbf{e}}_n(\lambda/\mu)$.
Thus $T=f_{j_1}^{l_1}f_{j_2}^{l_2}\cdots f_{j_t}^{l_t}(P) \in \SVT_{\mathbf{e}}(\lambda/\mu, \Phi;P)$ such that $\Psi(T)=\Tilde{T}$. It is easy to see that $T$ respects the flag $\Phi$ since the $i^{th}$ row $T_i$ of $T$ is inherited from $\Tilde{T}_i$ where $\Tilde{T}=\Tilde{T}_n * \cdots * \Tilde{T}_2 * \Tilde{T}_1$. Thus $\Psi$ is surjective.
\end{proof}
\begin{example}
\label{Ex: main Pro 1}
Let $\lambda=(2,2), \mu=(1,0)$ and $\Phi=(1,3),\mathbf{e}=(0,1)$. The only highest weight elements of $\SVT_{\mathbf{e}}(\lambda/\mu, \Phi)$ are the following:
$$
P= \begin{ytableau}
    \none &1\\1,2 &2   
\end{ytableau}
\hspace{1 cm}
Q=\begin{ytableau}
    \none &1\\1&2,3   
\end{ytableau}
$$
Now $\SVT_{\mathbf{e}}(\lambda/\mu, \Phi;P)$ contains the following three elements:
$$
\begin{ytableau}
    \none &1\\1,2 &2   
\end{ytableau}
\hspace{0.5 cm}
\begin{ytableau}
    \none &1\\1,2&3   
\end{ytableau}
\hspace{0.5 cm}
\begin{ytableau}
    \none &1\\1,3&3   
\end{ytableau}
$$
Also, the elements in $\SVT_{\mathbf{e}}(\lambda/\mu, \Phi;Q)$ are given below:
$$
\begin{ytableau}
    \none &1\\1&2,3   
\end{ytableau}
\hspace{0.5 cm}
\begin{ytableau}
    \none &1\\2&2,3   
\end{ytableau}
\hspace{0.5 cm}
\begin{ytableau}
    \none &1\\2,3&3  
\end{ytableau}
$$
Then $\SVT_{\mathbf{e}}(\lambda/\mu, \Phi;P), \SVT_{\mathbf{e}}(\lambda/\mu, \Phi;Q)$ are isomorphic to the Demazure crystals $\mathcal{B}_{s_2} (2,2,0),\mathcal{B}_{s_2s_1} (2,1,1)$ respectively, see Figure~\ref{fig:P}, Figure~\ref{fig:Q}.  
\begin{figure}
    \centering
\begin{tikzpicture}
    \draw (0,0)--(1.6,0)--(1.6,1.6);
    \draw (0,0)--(0,0.8)--(1.6,0.8);
    \draw (0.8,0)--(0.8,1.6)--(1.6,1.6);
    \node at (0.4,0.4) {$1,2$};
    \node at (1.2,0.4) {$2$};
    \node at (1.2,1.2) {$1$};
    \draw[blue][->] (0.8,-0.25) -- (0.8,-1.35);
    \node at (0.8+0.2,-0.8) {$\textcolor{blue}{2}$};
    \draw (0,0-3.3)--(1.6,0-3.3)--(1.6,1.6-3.3);
    \draw (0,0-3.3)--(0,0.8-3.3)--(1.6,0.8-3.3);
    \draw (0.8,0-3.3)--(0.8,1.6-3.3)--(1.6,1.6-3.3);
    \node at (0.4,0.4-3.3) {$1,2$};
    \node at (1.2,0.4-3.3) {$3$};
    \node at (1.2,1.2-3.3) {$1$};
    \draw[blue][->] (0.8,-0.25-3.3) -- (0.8,-1.35-3.3);
    \node at (0.8+0.2,-0.8-3.3) {$\textcolor{blue}{2}$};
    \draw (0,0-6.6)--(1.6,0-6.6)--(1.6,1.6-6.6);
    \draw (0,0-6.6)--(0,0.8-6.6)--(1.6,0.8-6.6);
    \draw (0.8,0-6.6)--(0.8,1.6-6.6)--(1.6,1.6-6.6);
    \node at (0.4,0.4-6.6) {$1,3$};
    \node at (1.2,0.4-6.6) {$3$};
    \node at (1.2,1.2-6.6) {$1$};
\end{tikzpicture}
\begin{tikzpicture}
    \draw (0,0)  node {$\null$};
    \draw[|->] (0.5,3.5) -- (1.5,3.5);
    \draw (1,3.8) node {$\Psi$};
    \draw (1.5+0.3,3.5) node {$\null$};    
\end{tikzpicture}
\begin{tikzpicture}
    \draw (0,0)--(0,1.6)--(2.4,1.6);
    \draw (0,0)--(0.8,0)--(0.8,1.6);
    \draw (0,0.8)--(1.6,0.8)--(1.6,2.4)--(2.4,2.4)--(2.4,1.6);
    \node at (0.4,0.4) {$2$};
    \node at (0.4,1.2) {$1$};
    \node at (1.2,1.2) {$2$};
    \node at (2,2) {$1$};
    \draw[blue][->] (0.4,-0.25) -- (0.4,-1.35);
    \node at (0.4+0.2,-0.8) {$\textcolor{blue}{2}$};
    \draw (0,0-3.3)--(0,1.6-3.3)--(2.4,1.6-3.3);
    \draw (0,0-3.3)--(0.8,0-3.3)--(0.8,1.6-3.3);
    \draw (0,0.8-3.3)--(1.6,0.8-3.3)--(1.6,2.4-3.3)--(2.4,2.4-3.3)--(2.4,1.6-3.3);
    \node at (0.4,0.4-3.3) {$2$};
    \node at (0.4,1.2-3.3) {$1$};
    \node at (1.2,1.2-3.3) {$3$};
    \node at (2,2-3.3) {$1$};
    \draw[blue][->] (0.4,-0.25-3.3) -- (0.4,-1.35-3.3);
    \node at (0.4+0.2,-0.8-3.3) {$\textcolor{blue}{2}$};
    \draw (0,0-6.6)--(0,1.6-6.6)--(2.4,1.6-6.6);
    \draw (0,0-6.6)--(0.8,0-6.6)--(0.8,1.6-6.6);
    \draw (0,0.8-6.6)--(1.6,0.8-6.6)--(1.6,2.4-6.6)--(2.4,2.4-6.6)--(2.4,1.6-6.6);
    \node at (0.4,0.4-6.6) {$3$};
    \node at (0.4,1.2-6.6) {$1$};
    \node at (1.2,1.2-6.6) {$3$};
    \node at (2,2-6.6) {$1$};
\end{tikzpicture}
\begin{tikzpicture}
    \draw (0,0)  node {$\null$};
    \draw[|->] (0.5,3.5) -- (1.5,3.5);
    \draw (1,3.8) node {$\rect$};
    \draw (1.5+0.3,3.5) node {$\null$};
\end{tikzpicture}
\begin{tikzpicture}
    \draw (0,0)--(0,1.6)--(1.6,1.6)--(1.6,0)--(0,0);
    \draw (0.8,0)--(0.8,1.6);
    \draw (0,0.8)--(1.6,0.8);
    \node at (0.4,0.4) {$2$};
    \node at (1.2,0.4) {$2$};
    \node at (1.2,1.2) {$1$};
    \node at (0.4,1.2) {$1$};
    \draw[blue][->] (0.8,-0.25) -- (0.8,-1.35);
    \node at (0.8+0.2,-0.8) {$\textcolor{blue}{2}$};
    \draw (0,0-3.3)--(0,1.6-3.3)--(1.6,1.6-3.3)--(1.6,0-3.3)--(0,0-3.3);
    \draw (0.8,0-3.3)--(0.8,1.6-3.3);
    \draw (0,0.8-3.3)--(1.6,0.8-3.3);
    \node at (0.4,0.4-3.3) {$2$};
    \node at (1.2,0.4-3.3) {$3$};
    \node at (1.2,1.2-3.3) {$1$};
    \node at (0.4,1.2-3.3) {$1$};
    \draw[blue][->] (0.8,-0.25-3.3) -- (0.8,-1.35-3.3);
    \node at (0.8+0.2,-0.8-3.3) {$\textcolor{blue}{2}$};
    \draw (0,0-6.6)--(0,1.6-6.6)--(1.6,1.6-6.6)--(1.6,0-6.6)--(0,0-6.6);
    \draw (0.8,0-6.6)--(0.8,1.6-6.6);
    \draw (0,0.8-6.6)--(1.6,0.8-6.6);
    \node at (0.4,0.4-6.6) {$3$};
    \node at (1.2,0.4-6.6) {$3$};
    \node at (1.2,1.2-6.6) {$1$};
    \node at (0.4,1.2-6.6) {$1$};
\end{tikzpicture}
    \caption{$\SVT_{\mathbf{e}}(\lambda/\mu, \Phi;P)$ is isomorphic to $\mathcal{B}_{s_2} (2,2,0)$}
    \label{fig:P}
\end{figure}
\begin{figure}
    \centering
\begin{tikzpicture}
    \draw (0,0)--(1.6,0)--(1.6,1.6);
    \draw (0,0)--(0,0.8)--(1.6,0.8);
    \draw (0.8,0)--(0.8,1.6)--(1.6,1.6);
    \node at (0.4,0.4) {$1$};
    \node at (1.2,0.4) {$2,3$};
    \node at (1.2,1.2) {$1$};
    \draw[blue][->] (0.8,-0.25) -- (0.8,-1.35);
    \node at (0.8+0.2,-0.8) {$\textcolor{blue}{1}$};
    \draw (0,0-3.3)--(1.6,0-3.3)--(1.6,1.6-3.3);
    \draw (0,0-3.3)--(0,0.8-3.3)--(1.6,0.8-3.3);
    \draw (0.8,0-3.3)--(0.8,1.6-3.3)--(1.6,1.6-3.3);
    \node at (0.4,0.4-3.3) {$2$};
    \node at (1.2,0.4-3.3) {$2,3$};
    \node at (1.2,1.2-3.3) {$1$};
    \draw[blue][->] (0.8,-0.25-3.3) -- (0.8,-1.35-3.3);
    \node at (0.8+0.2,-0.8-3.3) {$\textcolor{blue}{2}$};
    \draw (0,0-6.6)--(1.6,0-6.6)--(1.6,1.6-6.6);
    \draw (0,0-6.6)--(0,0.8-6.6)--(1.6,0.8-6.6);
    \draw (0.8,0-6.6)--(0.8,1.6-6.6)--(1.6,1.6-6.6);
    \node at (0.4,0.4-6.6) {$2,3$};
    \node at (1.2,0.4-6.6) {$3$};
    \node at (1.2,1.2-6.6) {$1$};
\end{tikzpicture}
\begin{tikzpicture}
    \draw (0,0)  node {$\null$};
    \draw[|->] (0.5,3.5) -- (1.5,3.5);
    \draw (1,3.8) node {$\Psi$};
    \draw (1.5+0.3,3.5) node {$\null$};    
\end{tikzpicture}
\begin{tikzpicture}
    \draw (0,0)--(0,1.6)--(2.4,1.6);
    \draw (0,0)--(0.8,0)--(0.8,1.6);
    \draw (0,0.8)--(1.6,0.8)--(1.6,2.4)--(2.4,2.4)--(2.4,1.6);
    \node at (0.4,0.4) {$3$};
    \node at (0.4,1.2) {$1$};
    \node at (1.2,1.2) {$2$};
    \node at (2,2) {$1$};
    \draw[blue][->] (0.4,-0.25) -- (0.4,-1.35);
    \node at (0.4+0.2,-0.8) {$\textcolor{blue}{1}$};
    \draw (0,0-3.3)--(0,1.6-3.3)--(2.4,1.6-3.3);
    \draw (0,0-3.3)--(0.8,0-3.3)--(0.8,1.6-3.3);
    \draw (0,0.8-3.3)--(1.6,0.8-3.3)--(1.6,2.4-3.3)--(2.4,2.4-3.3)--(2.4,1.6-3.3);
    \node at (0.4,0.4-3.3) {$3$};
    \node at (0.4,1.2-3.3) {$2$};
    \node at (1.2,1.2-3.3) {$2$};
    \node at (2,2-3.3) {$1$};
    \draw[blue][->] (0.4,-0.25-3.3) -- (0.4,-1.35-3.3);
    \node at (0.4+0.2,-0.8-3.3) {$\textcolor{blue}{2}$};
    \draw (0,0-6.6)--(0,1.6-6.6)--(2.4,1.6-6.6);
    \draw (0,0-6.6)--(0.8,0-6.6)--(0.8,1.6-6.6);
    \draw (0,0.8-6.6)--(1.6,0.8-6.6)--(1.6,2.4-6.6)--(2.4,2.4-6.6)--(2.4,1.6-6.6);
    \node at (0.4,0.4-6.6) {$3$};
    \node at (0.4,1.2-6.6) {$2$};
    \node at (1.2,1.2-6.6) {$3$};
    \node at (2,2-6.6) {$1$};
\end{tikzpicture}
\begin{tikzpicture}
    \draw (0,0)  node {$\null$};
    \draw[|->] (0.5,3.5) -- (1.5,3.5);
    \draw (1,3.8) node {$\rect$};
    \draw (1.5+0.3,3.5) node {$\null$};
\end{tikzpicture}
\begin{tikzpicture}[scale=0.85]
    \draw (0,0)--(0,2.4)--(1.6,2.4)--(1.6,1.6)--(0.8,1.6)--(0.8,0)--(0,0);
    \draw (0,0.8)--(0.8,0.8);
    \draw (0,1.6)--(0.8,1.6)--(0.8,2.4);
    \node at (0.4,0.4) {$3$};
    \node at (0.4,1.2) {$2$};
    \node at (0.4,2) {$1$};
    \node at (1.2,2) {$1$};
    \draw[blue][->] (0.8-0.4,-0.25) -- (0.8-0.4,-1.35);
    \node at (0.8+0.2-0.4,-0.8) {$\textcolor{blue}{1}$};
    \draw (0,0-4)--(0,2.4-4)--(1.6,2.4-4)--(1.6,1.6-4)--(0.8,1.6-4)--(0.8,0-4)--(0,0-4);
    \draw (0,0.8-4)--(0.8,0.8-4);
    \draw (0,1.6-4)--(0.8,1.6-4)--(0.8,2.4-4);
    \node at (0.4,0.4-4) {$3$};
    \node at (0.4,1.2-4) {$2$};
    \node at (0.4,2-4) {$1$};
    \node at (1.2,2-4) {$2$};
    \draw[blue][->] (0.8-0.4,-0.25-4) -- (0.8-0.4,-1.35-4);
    \node at (0.8+0.2-0.4,-0.8-4) {$\textcolor{blue}{2}$};
    \draw (0,0-8)--(0,2.4-8)--(1.6,2.4-8)--(1.6,1.6-8)--(0.8,1.6-8)--(0.8,0-8)--(0,0-8);
    \draw (0,0.8-8)--(0.8,0.8-8);
    \draw (0,1.6-8)--(0.8,1.6-8)--(0.8,2.4-8);
    \node at (0.4,0.4-8) {$3$};
    \node at (0.4,1.2-8) {$2$};
    \node at (0.4,2-8) {$1$};
    \node at (1.2,2-8) {$3$};
\end{tikzpicture}
    \caption{$\SVT_{\mathbf{e}}(\lambda/\mu, \Phi;Q)$ is isomorphic to $\mathcal{B}_{s_2s_1} (2,1,1)$}
    \label{fig:Q}
\end{figure}
\end{example}
\bremark
In general, $ch(\SVT_{\mathbf{e}}(\lambda/\mu,\Phi)) = \displaystyle\sum_{T \in \SVT_{\mathbf{e}}(\lambda/\mu,\Phi)}\characx^{\wt(T)}$ is not the flagged skew Schur polynomial $s_{\sigma_{\lambda/\mu}^{\mathbf{e}}}(X_{\Phi^{\mathbf{e}}})$. For example, if we assume $\lambda/\mu, \mathbf{e}, \Phi$ as in Example~\ref{Ex: main Pro 1}, then we have
$$ch(\SVT_{\mathbf{e}}(\lambda/\mu,\Phi))= \kappa_{(2,0,2)}+\kappa_{(1,1,2)}.$$
The crystal graph of $\Tab(\sigma_{\lambda/\mu}^{\mathbf{e}},\Phi^{\mathbf{e}}) - \Psi(\SVT_{\mathbf{e}}(\lambda/\mu, \Phi)),$ given in Figure~\ref{fig:flagged skew}, has one connected component, which is isomorphic to $\mathcal{B}_{s_2}(3,1,0)$. Thus 
$$s_{\sigma_{\lambda/\mu}^{\mathbf{e}}}(X_{\Phi^{\mathbf{e}}})=\kappa_{(2,0,2)}+\kappa_{(1,1,2)}+\kappa_{(3,0,1)}.$$
\eremark
\begin{figure}
    \centering
    \label{fig:flagged skew}
    \begin{tikzpicture}
    \draw (0,0)--(0,1.6)--(2.4,1.6);
    \draw (0,0)--(0.8,0)--(0.8,1.6);
    \draw (0,0.8)--(1.6,0.8)--(1.6,2.4)--(2.4,2.4)--(2.4,1.6);
    \node at (0.4,0.4) {$2$};
    \node at (0.4,1.2) {$1$};
    \node at (1.2,1.2) {$1$};
    \node at (2,2) {$1$};
    \draw[blue][->] (0.4,-0.25) -- (0.4,-1.35);
    \node at (0.4+0.2,-0.8) {$\textcolor{blue}{2}$};
    \draw (0,0-3.3)--(0,1.6-3.3)--(2.4,1.6-3.3);
    \draw (0,0-3.3)--(0.8,0-3.3)--(0.8,1.6-3.3);
    \draw (0,0.8-3.3)--(1.6,0.8-3.3)--(1.6,2.4-3.3)--(2.4,2.4-3.3)--(2.4,1.6-3.3);
    \node at (0.4,0.4-3.3) {$3$};
    \node at (0.4,1.2-3.3) {$1$};
    \node at (1.2,1.2-3.3) {$1$};
    \node at (2,2-3.3) {$1$};
\end{tikzpicture}
\begin{tikzpicture}
    \draw (0,0)  node {$\null$};
    \draw[|->] (0.5,3.5) -- (1.5,3.5);
    \draw (1,3.8) node {$\rect$};
    \draw (1.5+0.3,3.5) node {$\null$}; 
\end{tikzpicture}
\begin{tikzpicture}
    \draw (0,0)--(0,1.6)--(2.4,1.6);
    \draw (0,0)--(0.8,0)--(0.8,0.8)--(2.4,0.8)--(2.4,1.6);
    \draw (0,0.8)--(0.8,0.8)--(0.8,1.6);
    \draw (1.6,0.8)--(1.6,1.6);
    \node at (0.4,0.4) {$2$};
    \node at (0.4,1.2) {$1$};
    \node at (1.2,1.2) {$1$};
    \node at (2,1.2) {$1$};
    \draw[blue][->] (0.8,-0.25) -- (0.8,-1.35);
    \node at (0.8+0.2,-0.8) {$\textcolor{blue}{2}$};
    \draw (0,0-3.3)--(0,1.6-3.3)--(2.4,1.6-3.3);
    \draw (0,0-3.3)--(0.8,0-3.3)--(0.8,0.8-3.3)--(2.4,0.8-3.3)--(2.4,1.6-3.3);
    \draw (0,0.8-3.3)--(0.8,0.8-3.3)--(0.8,1.6-3.3);
    \draw (1.6,0.8-3.3)--(1.6,1.6-3.3);
    \node at (0.4,0.4-3.3) {$3$};
    \node at (0.4,1.2-3.3) {$1$};
    \node at (1.2,1.2-3.3) {$1$};
    \node at (2,1.2-3.3) {$1$};
\end{tikzpicture}
\caption{The crystal graph of $\Tab(\sigma_{\lambda/\mu}^{\mathbf{e}},\Phi^{\mathbf{e}}) - \Psi(\SVT_{\mathbf{e}}(\lambda/\mu, \Phi))$ for $\lambda,\mu,\mathbf{e},\Phi$ given in Example~\ref{Ex: main Pro 1}}
\end{figure}
Now using \eqref{eq:Demazure:excess}, Proposition~\ref{proposition:tableau} and Proposition~\ref{Proposition:main} we have the following theorem. 
\begin{theorem}
\label{theorem:main}
Let $\lambda/\mu$ $(\lambda, \mu \in \mathcal{P}[n])$ be a skew shape and $\Phi \in \mathcal{F}[n]$. Then 
$$ \SVT_{\mathbf{e}}(\lambda/\mu, \Phi)  \cong \displaystyle\bigsqcup_{R'} \mathcal{B}_{\tau} (\widehat{\beta(R')}^{\dagger},$$
where $R'$ runs over all $(\lambda/\mu, \Phi)$-compatible tableaux for SVT such that $\overline{\ex}(R')=\mathbf{e}$.     
\end{theorem}
\bremark
\label{remark:main}
Theorem~\ref{theorem:main} gives a Demazure crystal structure on $\SVT(\lambda/\mu, \Phi)$ since $$\SVT(\lambda/\mu, \Phi)=\displaystyle\bigsqcup_{\mathbf{e} \in \mathbb{Z}^n _{+}} \SVT_{\mathbf{e}}(\lambda/\mu, \Phi) $$
\eremark
\begin{corollary}
\begin{upshape}
    \cite[Theorem 3.11 \& Appendix]{KRSV}
\end{upshape}
\label{corollary:SVT-tableau}
Since $\Tab(\lambda/\mu, \Phi)=\SVT_{\mathbf{0}}(\lambda/\mu, \Phi)$ we have the following
   $$ \Tab(\lambda/\mu, \Phi)  \cong \displaystyle\bigsqcup_{R'} \mathcal{B}_{\tau} (\widehat{\beta(R')}^{\dagger},$$
where $R'$ runs over all $(\lambda/\mu, \Phi)$-compatible tableaux for SVT such that $\overline{\ex}(R')=\mathbf{0}$.     
\end{corollary}
\begin{corollary}
\label{corollary:main}
Let $\lambda/\mu$ $(\lambda, \mu \in \mathcal{P}[n])$ be a skew shape and $\Phi \in \mathcal{F}[n]$. Then
$$ G_{\lambda/\mu}(X_{\Phi};\mathbf{t})=\displaystyle\sum_{\alpha \in \mathbb{Z}^n _{+}}(-1)^{|\alpha|}\mathbf{t}^{\alpha}\displaystyle\sum_{R'}\kappa_{\widehat{\beta(R')}},$$
where $R'$ runs over all $(\lambda/\mu, \Phi)$-compatible tableaux for SVT such that $\overline{\ex}(R')=\alpha$.

At $\mathbf{t}=\mathbf{0},$ we obtain the expansion of the flagged skew Schur polynomial $s_{\lambda/\mu}(X_{\Phi})$ in terms of key polynomials, given by Reiner and Shimozono \cite[Theorem 20]{RS}, as follows
$$ s_{\lambda/\mu}(X_{\Phi})=G_{\lambda/\mu}(X_{\Phi};\mathbf{0})=\displaystyle\sum_{R'}\kappa_{\widehat{\beta(R')}},$$
where $R'$ runs over all $(\lambda/\mu, \Phi)$-compatible tableaux for SVT such that $\overline{\ex}(R')=\mathbf{0}$.
\end{corollary}
\begin{example}
    Let $\lambda=(2,2), \mu=(1,0)$ and $\Phi=(1,3)$. Then the highest weight elements of $\SVT(\lambda/\mu, \Phi)$ are the following:
$$ P =\ytableausetup{mathmode,
notabloids}
\begin{ytableau}
    \none &1\\ 1 &2   
\end{ytableau}
\hspace{0.5 cm}
Q =\ytableausetup{mathmode,
notabloids}
\begin{ytableau}
    \none &1\\1,2 &2   
\end{ytableau}
\hspace{0.5 cm}
R =\ytableausetup{mathmode,
notabloids}
\begin{ytableau}
    \none &1\\1&2,3   
\end{ytableau}
\hspace{0.5 cm}
S =\ytableausetup{mathmode,
notabloids}
\begin{ytableau}
    \none &1\\1,2 &2,3   
\end{ytableau}
$$
Then
$$ 
\Tilde{P} =\ytableausetup{mathmode,
notabloids}
\begin{ytableau}
    \none & \none &1\\ 1 &2   
\end{ytableau}
\hspace{0.5 cm}
\Tilde{Q} =\ytableausetup{mathmode,
notabloids}
\begin{ytableau}
    \none & \none &1\\1 &2 \\ 2   
\end{ytableau}
\hspace{0.5 cm}
\Tilde{R} =\ytableausetup{mathmode,
notabloids}
\begin{ytableau}
    \none & \none &1\\1&2 \\ 3   
\end{ytableau}
\hspace{0.5 cm}
\Tilde{S} =\ytableausetup{mathmode,
notabloids}
\begin{ytableau}
    \none & \none &1\\1 &2 \\ 2\\3  
\end{ytableau}
$$
It is clear that $\Tilde{P}$ respects the flag $\Phi^{(0,0)}=(1,3),$ $\Tilde{Q}$ respects the flag $\Phi^{(0,1)}=(1,3,3),$ $\Tilde{R}$ respects the flag $\Phi^{(0,1)}=(1,3,3)$ and $\Tilde{S}$ respects the flag $\Phi^{(0,2)}=(1,3,3,3)$.

Now
$$(
\begin{bmatrix}
           \mathbf{b}(\sigma_{\lambda/\mu}^{(0,0)}) \\
           r_{\Tilde{P}} \\
           
\end{bmatrix} \rightarrow \emptyset)
=(\ytableausetup{mathmode,
notabloids}
  \begin{ytableau}
    1&1\\2   
  \end{ytableau},\ytableausetup{mathmode,
notabloids}
  \begin{ytableau}
    1&2\\2   
  \end{ytableau} );
  \hspace{1 cm}
(
\begin{bmatrix}
           \mathbf{b}(\sigma_{\lambda/\mu}^{(0,1)}) \\
           r_{\Tilde{Q}} \\
\end{bmatrix} \rightarrow \emptyset)
=(\ytableausetup{mathmode,
notabloids}
  \begin{ytableau}
    1&1\\2&2   
  \end{ytableau},\ytableausetup{mathmode,
notabloids}
  \begin{ytableau}
    1&2\\2 &3
  \end{ytableau} ); 
$$
$$
(
\begin{bmatrix}
           \mathbf{b}(\sigma_{\lambda/\mu}^{(0,1)}) \\
           r_{\Tilde{R}} \\
           
\end{bmatrix} \rightarrow \emptyset)
=(\ytableausetup{mathmode,
notabloids}
  \begin{ytableau}
    1&1\\2\\3   
  \end{ytableau},\ytableausetup{mathmode,
notabloids}
  \begin{ytableau}
    1&2\\2\\3   
  \end{ytableau} );
\hspace{1 cm}  
(
\begin{bmatrix}
           \mathbf{b}(\sigma_{\lambda/\mu}^{(1,1)}) \\
           r_{\Tilde{S}} \\
\end{bmatrix} \rightarrow \emptyset)
=(\ytableausetup{mathmode,
notabloids}
  \begin{ytableau}
    1&1\\2&2 \\3  
  \end{ytableau},\ytableausetup{mathmode,
notabloids}
  \begin{ytableau}
    1&2\\2 &3 \\4
  \end{ytableau} ).
$$
Then the set of all $(\lambda/\mu,\Phi)$-compatible tableaux for SVT contains the following tableaux:
$$ \hat{P} =\ytableausetup{mathmode,
notabloids}
\begin{ytableau}
     1 &2 \\2   
\end{ytableau}
\text{ with } \overline{\ex}(\hat{P})=(0,0,0);
\hspace{0.5 cm}
\hat{Q} =\ytableausetup{mathmode,
notabloids}
\begin{ytableau}
    1&2 \\ 2&3   
\end{ytableau}
\text{ with } \overline{\ex}(\hat{Q})=(0,1,0);
$$ 
$$
\hat{R} =\ytableausetup{mathmode,
notabloids}
\begin{ytableau}
    1&2\\ 2 \\3   
\end{ytableau}
\text{ with } \overline{\ex}(\hat{R})=(0,1,0);
\hspace{0.5 cm}
\hat{S} =\ytableausetup{mathmode,
notabloids}
\begin{ytableau}
    1&2 \\2&3 \\ 4  
\end{ytableau}
\text{ with } \overline{\ex}(\hat{S})=(0,2
,0).
$$
Then the corresponding left-key tableaux (using \cite{Willis:left-key}, \cite{Mrigendra:left-key}) are the following:
$$ K_{\_}(\hat{P}) =\ytableausetup{mathmode,
notabloids}
\begin{ytableau}
     1 &2 \\2   
\end{ytableau}
\text{ with } \beta(\hat{P})=(1,2,0);
\hspace{0.5 cm}
K_{\_}(\hat{Q}) =\ytableausetup{mathmode,
notabloids}
\begin{ytableau}
    1&1 \\ 2&2   
\end{ytableau}
\text{ with } \beta(\hat{Q})=(2,2,0);
$$ 
$$
K_{\_}(\hat{R}) =\ytableausetup{mathmode,
notabloids}
\begin{ytableau}
    1&2\\ 2 \\3   
\end{ytableau}
\text{ with } \beta(\hat{R})=(1,2,1);
\hspace{0.5 cm}
K_{\_}(\hat{S}) =\ytableausetup{mathmode,
notabloids}
\begin{ytableau}
    1&1 \\2&2 \\ 4  
\end{ytableau}
\text{ with } \beta(\hat{S})=(2,2,0,1).
$$
Now using Theorem~\ref{theorem:RS-hat} we have the following:
$$\beta(\hat{P})=(1,2,0), \Phi^{(0,0)}=(1,3,3) \implies \widehat{\beta(\hat{P})}=(1,0,2);$$
$$\beta(\hat{Q})=(2,2,0), \Phi^{(0,1)}=(1,3,3) \implies \widehat{\beta(\hat{Q})}=(2,0,2);$$ 
$$\beta(\hat{R})=(1,2,1), \Phi^{(0,1)}=(1,3,3) \implies \widehat{\beta(\hat{R})}=(1,1,2);$$
$$\beta(\hat{S})=(2,2,0,1), \Phi^{(0,2)} =(1,3,3,3) \implies\widehat{\beta(\hat{S})}=(2,1,2).$$
Therefore, we have the following expansion
$$ G_{\lambda/\mu}(X_{\Phi};\mathbf{t})=\kappa_{(1,0,2)} - t_2 \kappa_{(2,0,2)} -t_2 \kappa_{(1,1,2)} + t_2^2 \kappa_{(2,1,2)}.$$
\end{example}
\begin{corollary}
\label{corollary:schur-SVT}
Let $\lambda/\mu$ $(\lambda, \mu \in \mathcal{P}[n])$ be a skew shape and $\Phi=(n,n,\ldots,n) = n^n$ (say). Then 
$$ RG_{\lambda/\mu}(\characx;\mathbf{t})=\displaystyle\sum_{\alpha \in \mathbb{Z}^n _{+}}(-1)^{|\alpha|}\mathbf{t}^{\alpha}\displaystyle\sum_{R'}s_{\widehat{\beta(R')}^{\dagger}}$$
$$ \hspace{2.7 cm} =\displaystyle\sum_{\alpha \in \mathbb{Z}^n _{+}}(-1)^{|\alpha|}\mathbf{t}^{\alpha}\displaystyle\sum_{R'}s_{\sh(R')},$$
where $R'$ runs over all $(\lambda/\mu, n^n)$-compatible tableaux for SVT such that $\overline{\ex}(R')=\alpha$.
\end{corollary}
\section{various expansions}
\label{Section 5}
In this section, we provide expansions of the refined dual stable Grothendieck polynomial $\widetilde{g}_{\lambda / \mu}(\characx;\mathbf{t})$ in terms of the stable Grothendieck polynomials $RG_{\lambda}(\mathbf{x};\mathbf{1})$ and in terms of the dual stable Grothendieck polynomials $g_{\lambda}(\characx)$. We also give similar expansions for the row-refined skew stable Grothendieck polynomial $RG_{\lambda/\mu}(\characx,\mathbf{t})$ and the Schur P-functions $P_{\lambda}(\characx)$. We will give those expansions using methods established in \cite{Morse-Jason:K-bases}. The similar expansions for canonical Grothendieck polynolials have been given in \cite{Anne:un-hook}.

\begin{definition}
\label{definition:rpp}
A \emph{reverse plane partition} (RPP) of skew shape $\lambda/\mu$ is a filling of the skew diagram $\lambda/\mu$ with positive integers which is weakly increasing along both rows and columns.

The \emph{weight} of $R$ is defined as $\wt(R):=(r_1,r_2,\ldots),$ where $r_i$ is the number of columns of $R$ containing $i$.

For a RPP $R,$ we circle only the topmost occurrence of each letter in each column of $R$. The \emph{reading word} of a reverse plane partition $R,$ denoted by $w(R)$, is the word $w_1w_2\cdots$ obtained by reading the circled elements of $R$ starting from the bottom row, left to right and then continuing up the rows.
\end{definition}
\begin{example}
\label{example:rpp}
$R=$ $  
\ytableausetup{mathmode,
notabloids}
  \begin{ytableau}
    1&2&3\\1&3   
  \end{ytableau}
$ is a reverse plane partition of shape $(3,2,0)$ with  $\wt(R)=(1,1,2)$ and $w(R)=3123$.
\end{example}

Lam and Pylyavskyy \cite{Lam} have given the following formula for the dual stable Grothendieck polynomial $g_{\lambda/\mu}(\characx)$ for the skew shape $\lambda/\mu$ 
$$ g_{\lambda/\mu}(\characx) = \displaystyle\sum_{R \in \RPP_n(\lambda/\mu)} \characx^{\wt(R)},$$
where $\RPP_n(\lambda/\mu)$ denotes the set of all reverse plane partitions of shape $\lambda/\mu$ with entries at most $n$.

As per \cite{Morse-Jason:K-bases}, a symmetric function, $f_{\alpha}$, is said to have a \emph{tableaux Schur expansion} if there is a set of semi-standard tableaux $\mathbb{T}(\alpha)$ and a weight function $\wt_{\alpha}$ such that $$ f_{\alpha}=\displaystyle\sum_{T \in \mathbb{T}(\alpha)} \wt_{\alpha}(T)s_{\sh(T)}(\mathbf{x})$$

Given $\mathbb{T}(\alpha),$ let $\mathbb{S}(\alpha), \mathbb{R}(\alpha)$ be the sets of set-valued tableaux, reverse plane partitions of partition shapes respectively, defined as follows: 
$$ S \in \mathbb{S}(\alpha) \text{ if and only if } \rect(w(S)) \in \mathbb{T}(\alpha), $$
$$ R \in \mathbb{R}(\alpha) \text{ if and only if } \rect(w(R)) \in \mathbb{T}(\alpha).$$
Similarly, we also extend $\wt_{\alpha}$ to $\mathbb{S}(\alpha)$ and $\mathbb{R}(\alpha)$ by $\wt_{\alpha}(X):=\wt_{\alpha}(\rect(w(X)))$. Here, given any word $\mathbf{u},$ $\rect(\mathbf{u})$ denotes the unique semi-standard tableau Knuth equivalent to the word $\mathbf{u}$.

Any symmetric function with that property has the following expansion in terms of $g_{\lambda}, G_{\lambda}$. 
\begin{theorem}\begin{upshape}
    \cite[Theorem 3.5]{Morse-Jason:K-bases}
\end{upshape}
\label{theorem:K bases}
Given $$f_{\alpha} = \displaystyle\sum_{T \in \mathbb{T}(\alpha)} \wt_{\alpha}(T)s_{\sh(T)}(\mathbf{x}),$$ 
we have $$f_{\alpha} = \displaystyle\sum_{R \in \mathbb{R}(\alpha)} \wt_{\alpha}(R)RG_{\sh(R)}(\mathbf{x};\mathbf{1})$$
$$ \hspace{1.4 cm}= \displaystyle\sum_{S \in \mathbb{S}(\alpha)} (-1)^{|\ex(S)|}\wt_{\alpha}(S) g_{\sh(S)}(\mathbf{x})$$
\end{theorem}
\subsection{The refined dual stable Grothendieck polynomial} Given $\lambda, \mu \in \mathcal{P}[n]$ and a flag $\Phi \in \mathcal{F}[n] ,$ $ \RPP(\lambda/\mu,\Phi)$ denotes the set of all reverse plane partitions $R$ of shape $\lambda/\mu$ such that the entries in $i^{th}$ row of $R$ are $\leq \Phi_i $ for all $1 \leq i \leq n$.
Following Galashin, Grinberg and Liu \cite{Galashin:Bender-Knuth}, we define the column equalities vector (in short, the \textit{ceq} statistic) of a reverse plane partition $R$ in $\RPP_n(\lambda/\mu)$ as $\ceq(R):=(c_1,c_2,\ldots) \in \mathbb{Z}^n _{+}$ s.t. $c_i$ is the number of boxes $(i,j)$ s.t. $(i,j),(i+1,j) \in \lambda/\mu$ and $R(i,j)=R(i+1,j)$. For instance, if $R$ is the reverse plane partition, given below,
$$ \ytableausetup{mathmode,
notabloids}
  \begin{ytableau}
    \none &2 &3 &4\\ 1 &2 &3 \\ 1 &3
  \end{ytableau}
$$
then $\ceq(R)=(2,1,0)$.

We define the \textit{flagged refined dual stable Grothendieck polynomial} $$g_{\lambda / \mu}(X_{\Phi};\mathbf{t}): = \displaystyle\sum _{R \in \RPP(\lambda / \mu, \Phi)} \mathbf{t}^{\ceq(R)} \characx^{\wt(R)}$$
$g_{\lambda / \mu}(X_{\Phi};\mathbf{t})$ is the row-flaged refined dual stable Grothendieck polynomial $\widetilde{g}_{\lambda / \mu}^{~row(\mathbf{1},\Phi)}(\characx;\mathbf{t})$ in \cite[\S2.2]{Kim}.
\begin{definition}
\label{definition:row reading}
The \emph{row reading word} of a reverse plane partition $T$, denoted by $r_T$, is obtained as follows: omit all entries from $T$ which are equal to the entry immediately below it; then read all the remaining entries row-by-row, starting from bottom row, left to right and continuing up the rows.

The \emph{height} $h(T)$ of $T$ is defined as the sequence of positive integers whose $i^{th}$ part (from the left) is the row number of the $i^{th}$ letter (from the left) in $r_T$. Clearly, $h(T)$ is a weakly decreasing sequence.
\end{definition}
\bremark
The reading word $w(T)$ in Definition~\ref{definition:rpp} and row reading word $r_T$ in Definition~\ref{definition:row reading} of a reverse plane partition $T$ are not same. 
\eremark
\begin{example}
The row reading word and height of the reverse plane partition in Example~\ref{example:rpp} are $1323, 2211$ respectively. 
\end{example}
We recall that a semi-standard tableau $Q$ is said to be $(\lambda/\mu , \Phi)$-compatible for RPP if $\exists $ a unique $T' \in  \RPP(\lambda/\mu,\Phi)$ such that $r_{T'}$ is a Yamanouchi word together with $(
\begin{bmatrix}
           h(T) \\
           r_{T'} \\
           
\end{bmatrix} \rightarrow \emptyset)
=(-, Q) 
$, see \S4 in \cite{Sidhu} for more details. Then we define $\overline{\ceq}(Q):=\ceq(T')$.
\begin{theorem}
\begin{upshape}\cite[Theorem 1]{Sidhu}
\end{upshape}
$\RPP(\lambda/\mu,\Phi) $ is a disjoint union of Demazure crystals (up to isomorphism). More precisely,
$$\RPP(\lambda/\mu,\Phi) \cong \displaystyle\bigsqcup_{Q} \mathcal{B}_{\tau} (\widehat{\beta(Q)}^{\dagger}),$$
where $Q$ varies over all $(\lambda/\mu,\Phi)$-compatible tableaux for RPP and $\tau$ is any permutation s.t. $\tau. \widehat{\beta(Q)}^{\dagger}=\widehat{\beta(Q)}$.
\end{theorem}
Since the operators $e_i$ and $f_i$ preserve the $\ceq$ statistic (see Remark 3, \S2 in \cite{Galashin:LR}), we have the following result.
\begin{corollary}
$g_{\lambda / \mu}(X_{\Phi};\mathbf{t}) = \displaystyle\sum_{\alpha}\mathbf{t}^\alpha \displaystyle\sum_{Q} \key_{\widehat{\beta(Q)}},$
where $\alpha$ is over all $n$-tuple in $\mathbb{Z}^n_{+}$ and $Q$ runs over all $(\lambda/\mu,\Phi)$-compatible tableaux for RPP s.t. $\overline{\ceq}(Q)=\alpha$.     
\end{corollary}
\begin{corollary}
If $\Phi=(n,n,\ldots,n)=n^n$ (say) then $g_{\lambda / \mu}(X_{\Phi};\mathbf{t}) = \displaystyle\sum_{\alpha \in \mathbb{Z}^n_{+}}\mathbf{t}^\alpha \displaystyle\sum_{Q} 
s_{\sh(Q)}(\characx),$ where $Q$ varies over all $(\lambda/\mu,n^n)$-compatible tableaux for RPP s.t. $\overline{\ceq}(Q)=\alpha$. Thus in this case, $g_{\lambda / \mu}(X_{\Phi};\mathbf{t})$ is the refined dual stable Grothendieck polynomial $\widetilde{g}_{\lambda / \mu}(\characx;\mathbf{t})$ in \cite[\S3]{Galashin:Bender-Knuth}, \cite[Remark 3]{Galashin:LR}.
\end{corollary}
\bremark
The Hall inner product on the ring of symmetric polynomials is defined by $\langle s_{\lambda}, s_{\mu} \rangle =\delta_{\lambda \mu}$. Grinberg has indicated that $RG_{\lambda}(\characx,\mathbf{t}), \widetilde{g}_{\lambda}(\characx;\mathbf{t})$ are dual with respect to the Hall inner product, i.e., 
$$\langle RG_{\lambda}(\characx,\mathbf{t}), \widetilde{g}_{\lambda}(\characx;\mathbf{t}) \rangle = \delta_{\lambda \mu}$$
see \cite[Remark 3.9]{Melody-Nathan}, \cite[Proposition 2.6]{Travis:Schur}.
\eremark
Let $\mathbb{T}(\lambda/\mu)$ be the set of $ (\lambda/\mu,n^n)$-compatible tableau for RPP and $\wt_{\lambda/\mu}(Q)= \mathbf{t}^{\overline{\ceq}(Q
)},$ for $Q \in \mathbb{T}(\lambda/\mu)$. Then
\begin{equation}
\label{expansion:RPP}
 \widetilde{g}_{\lambda / \mu}(\characx;\mathbf{t})= \displaystyle\sum_{Q \in \mathbb{T}(\lambda/\mu)}\wt_{\lambda/\mu}(Q)s_{\sh(Q)}(\characx),  
\end{equation}
Thus Eq.~\ref{expansion:RPP} gives the tableaux Schur expansion for $\widetilde{g}_{\lambda / \mu}(\characx;\mathbf{t})$.
Therefore using Theorem~\ref{theorem:K bases} we have the following
$$\widetilde{g}_{\lambda / \mu}(\characx;\mathbf{t}) = \displaystyle\sum_{R \in \mathbb{R}(\lambda/\mu)} \wt_{\lambda/\mu}(R)RG_{\sh(R)}(\characx,\mathbf{1})$$
$$ \hspace{2.8cm}= \displaystyle\sum_{S \in \mathbb{S}(\lambda/\mu)} (-1)^{|\ex(S)|} \wt_{\lambda/\mu}(S) g_{\sh(S)}(\characx),$$
where
$$\mathbb{R}(\lambda/\mu) = \{ R : R \text{ is a RPP such that } w(R) \text{ is Knuth equivalent to an element of } \mathbb{T}(\lambda/\mu)  \}$$
$$\mathbb{S}(\lambda/\mu) =\{ S : S \text{ is a set-valued tableau and $w(S)$ is Knuth equivalent to an element of } \mathbb{T}(\lambda/\mu) \}$$
\begin{example}
Let $\lambda=(3,2,0),\mu=(1,0,0)$. Then the set of all reverse plane partitions of shape $\lambda/\mu$ such that their row reading words are Yamanouchi, contains the following three RPPs:
$$ R_1 =\ytableausetup{mathmode,
notabloids}
  \begin{ytableau}
    \none &1&1\\ 1 &2   
  \end{ytableau}
  \hspace{0.4 cm}
 R_2 =\ytableausetup{mathmode,
notabloids}
  \begin{ytableau}
    \none &1&1\\2 &2   
  \end{ytableau}
  \hspace{0.4 cm}
  R_3=\ytableausetup{mathmode,
notabloids}
  \begin{ytableau}
    \none &1 & 1\\1 &1   
  \end{ytableau}$$
It is easy to see that $\ceq(R_1)=(0,0,0), \ceq(R_2)=(0,0,0), \ceq(R_3)=(1,0,0)$.
$$(
\begin{bmatrix}
           h(R_1) \\
           r_{R_1} \\
           
\end{bmatrix} \rightarrow \emptyset)
=(\ytableausetup{mathmode,
notabloids}
  \begin{ytableau}
    1&1&1\\2   
  \end{ytableau},\ytableausetup{mathmode,
notabloids}
  \begin{ytableau}
    1&1&2\\2   
  \end{ytableau} ) 
$$
$$(
\begin{bmatrix}
           h(R_2) \\
           r_{R_2} \\
           
\end{bmatrix} \rightarrow \emptyset)
=(\ytableausetup{mathmode,
notabloids}
  \begin{ytableau}
    1&1\\2&2   
  \end{ytableau},\ytableausetup{mathmode,
notabloids}
  \begin{ytableau}
    1&1\\2 &2
  \end{ytableau} ) 
$$
$$(
\begin{bmatrix}
           h(R_3) \\
           r_{R_3} \\
           
\end{bmatrix} \rightarrow \emptyset)
=(\ytableausetup{mathmode,
notabloids}
\begin{ytableau}
    1&1&1   
\end{ytableau},\ytableausetup{mathmode,
notabloids}
\begin{ytableau}
    1&2&2  
\end{ytableau} ) 
$$
\vspace{0.2cm}

Thus $\mathbb{T}(\lambda/\mu)=\{
\ytableausetup{mathmode,
notabloids}
  $\begin{ytableau}
    1&1&2\\2   
  \end{ytableau}$,
  \hspace{0.1cm}
 $\begin{ytableau}
   1&1\\2&2   
  \end{ytableau}$,
  \hspace{0.1cm}
  $\begin{ytableau}
   1&2&2   
  \end{ytableau}$
\}$. 
By assigning labels $T_1, T_2, T_3$ (corresponds to $R_1,R_2,R_3$ respectively) to the tableaux of $\mathbb{T}(\lambda/\mu) $, we obtain the following
$$ \wt_{\lambda/\mu}(T_1)=1,\wt_{\lambda/\mu}(T_2)=1,\wt_{\lambda/\mu}(T_3)=t_1.$$
\vspace{0.1cm}
$\{
R \in \mathbb{R}(\lambda/\mu) | \rect(w(R))=T_1
\}=
\{
\begin{ytableau}
   1&1&2\\2   
\end{ytableau},
\begin{ytableau}
   1&1&2\\1&2   
  \end{ytableau},
\begin{ytableau}
   1&1&2\\1&2&2   
  \end{ytableau},\dots
 \}$
 
\vspace{0.2cm}
$\{
R \in \mathbb{R}(\lambda/\mu) | \rect(w(R))=T_2
\}
=\{
\begin{ytableau}
   1&1\\2&2   
  \end{ytableau},
\begin{ytableau}
   1&1\\1&2\\2   
  \end{ytableau},
\begin{ytableau}
   1&1\\2&2\\2   
\end{ytableau},\dots  
\}$

\vspace{0.2cm}
$\{
R \in \mathbb{R}(\lambda/\mu) | \rect(w(R))=T_3
\}
=\{
\begin{ytableau}
   1&2&2   
\end{ytableau},  
\begin{ytableau}
   1&2&2\\1   
\end{ytableau},  
\begin{ytableau}
   1&2&2\\1&2   
\end{ytableau},\dots  
\}
$
\vspace{0.3 cm}

Thus
$\widetilde{g}_{\lambda / \mu}(\characx;\mathbf{t})=(RG_{(3,1,0)}(\characx,\mathbf{1})+RG_{(3,2,0)}(\characx,\mathbf{1})+RG_{(3,3,0)}(\characx,\mathbf{1})+\cdots) + (RG_{(2,2,0)}(\characx,\mathbf{1})+2RG_{(2,2,1)}(\characx,\mathbf{1})+\cdots)+t_1(RG_{(3,0,0)}(\characx,\mathbf{1})+RG_{(3,1,0)}(\characx,\mathbf{1})+RG_{(3,2,0)}(\characx,\mathbf{1})+\cdots)$.

\vspace{0.3 cm}
$\{
S \in \mathbb{S}(\lambda/\mu) | \rect(w(S))=T_1
\} = \{\begin{ytableau}
    1&1&2\\2   
  \end{ytableau},
  \begin{ytableau}
    1& {1,2} &2   
  \end{ytableau}
  \}$

\vspace{0.2cm}
$\{
S \in \mathbb{S}(\lambda/\mu) | \rect(w(S))=T_2
\}
=\{
\begin{ytableau}
   1&1\\2&2   
\end{ytableau},
\begin{ytableau}
   1&1,2\\2   
\end{ytableau}
\}$

\vspace{0.2cm}
$\{
S \in \mathbb{S}(\lambda/\mu) | \rect(w(S))=T_3
\}
=\{\begin{ytableau}
   1&2&2   
\end{ytableau}
\}$

Therefore, 
$\widetilde{g}_{\lambda / \mu}(\characx;\mathbf{t}) = 
g_{(3,1,0)}(\characx)-g_{(3,0,0)}(\characx)+g_{(2,2,0)}(\characx)-g_{(2,1,0)}(\characx)+t_1g_{(3,0,0)}(\characx)$.
\end{example}
\subsection{The row-refined skew stable Grothendieck polynomial} Now we recall the row-refined skew stable Grothendieck polynomial is defined as  
$$RG_{\lambda/\mu}(\characx; \mathbf{t}):= \displaystyle \sum _{T \in \SVT_n(\lambda/\mu)} (-1)^{|\ex(T)|}\mathbf{t}^{\ex(T)}\characx^{\wt(T)} $$
For $\alpha \in \mathbb{Z}^n_{+},$ $\mathbb{T}(\lambda/\mu, \alpha)$ is the set of all $(\lambda/\mu,n^n)$-compatible tableaux for SVT with $\overline{\ex}(P)=\alpha$ $ \forall P \in \mathbb{T}(\lambda/\mu, \alpha)$. We also define $\wt_{\lambda/\mu, \alpha}(P):=(-1)^{|\alpha|}\mathbf{t}^{\alpha}$ $ \forall P \in \mathbb{T}(\lambda/\mu, \alpha)$. Then Corollary~\ref{corollary:schur-SVT} gives the following expansion
\begin{equation}
\label{set value:eq}
RG_{\lambda/\mu}(\characx; \mathbf{t}) = \displaystyle\sum_{\alpha \in \mathbb{Z}^n_{+}}\displaystyle\sum_{P\in \mathbb{T}(\lambda/\mu, \alpha)}\wt_{\lambda/\mu, \alpha}(P)s_{\sh(P)}(\characx)
\end{equation}

Thus \eqref{set value:eq} gives a tableau Schur expansion of $RG_{\lambda/\mu}(\characx; \mathbf{t})$. So by Theorem~\ref{theorem:K bases} we have the following.
$$RG_{\lambda/\mu}(\characx; \mathbf{t})= \displaystyle\sum_{\alpha \in \mathbb{Z}^n_{+}} \displaystyle \sum_{R \in \mathbb{R}(\lambda/\mu , \alpha)} \wt_{\lambda/\mu, \alpha}(R)RG_{\sh(R)}(\characx,\mathbf{1})$$
$$\hspace{2.9 cm}=\displaystyle\sum_{\alpha \in \mathbb{Z}^n_{+}}\displaystyle \sum_{S \in \mathbb{S}(\lambda/\mu, \alpha)} (-1)^{|\ex(S)|} \wt_{\lambda/\mu, \alpha}(S)g_{\sh(S)}(\characx)$$
where $\mathbb{R}(\lambda/\mu, \alpha)$ is the set of all reverse plane partitions whose reading word is Knuth equivalent to an element of $\mathbb{T}(\lambda/\mu, \alpha)$ and $\mathbb{S}(\lambda/\mu, \alpha)$ is the set of all set-valued tableaux whose reading word is Knuth equivalent to an element of $\mathbb{T}(\lambda/\mu, \alpha)$.
\begin{example}
Let $\lambda=(2,2), \mu=(1)$. Then the highest weight elements of $\SVT_2(\lambda/\mu)$ are the following:
$$
P=\ytableausetup{mathmode,
notabloids}
\begin{ytableau}
    \none &1 \\ 1 &2   
\end{ytableau}
  \hspace{1 cm}
Q=\ytableausetup{mathmode,
notabloids}
\begin{ytableau}
    \none &1\\1,2 &2   
\end{ytableau}$$
Then
$$ \Tilde{P}=\ytableausetup{mathmode,
notabloids}
\begin{ytableau}
    \none & \none &1 \\ 1 &2   
\end{ytableau}
  \hspace{1 cm}
\Tilde{Q} =\ytableausetup{mathmode,
notabloids}
\begin{ytableau}
    \none & \none &1\\1 &2 \\ 2   
\end{ytableau}$$
Now
$$(
\begin{bmatrix}
           \mathbf{b}(\sigma_{\lambda/\mu}^{(0,0)}) \\
           r_{\Tilde{P}} \\
           
\end{bmatrix} \rightarrow \emptyset)
=(\ytableausetup{mathmode,
notabloids}
  \begin{ytableau}
    1&1\\2   
  \end{ytableau},\ytableausetup{mathmode,
notabloids}
  \begin{ytableau}
    1&2\\2   
  \end{ytableau} ); 
\hspace{1 cm}
(
\begin{bmatrix}
           \mathbf{b}(\sigma_{\lambda/\mu}^{(0,1)}) \\
           r_{\Tilde{Q}} \\
\end{bmatrix} \rightarrow \emptyset)
=(\ytableausetup{mathmode,
notabloids}
  \begin{ytableau}
    1&1\\2&2   
  \end{ytableau},\ytableausetup{mathmode,
notabloids}
  \begin{ytableau}
    1&2\\2 &3
  \end{ytableau} ) 
$$
Let $\alpha=(0,0), \beta=(0,1)$. Then $\mathbb{T}(\lambda/\mu, \alpha)=\{
\ytableausetup{mathmode,
notabloids}
  \begin{ytableau}
    1&2\\2   
  \end{ytableau}
\}$;
$\mathbb{T}(\lambda/\mu, \beta)=\{
\ytableausetup{mathmode,
notabloids}
 $\begin{ytableau}
   1&2\\2&3   
  \end{ytableau}$
\}$.
So we obtain the following
$$ \wt_{\lambda/\mu, \alpha}(\ytableausetup{mathmode,
notabloids}
 \begin{ytableau}
   1&2\\2   
  \end{ytableau}
)=1,\wt_{\lambda/\mu,\beta}(\ytableausetup{mathmode,
notabloids}
 \begin{ytableau}
   1&2\\2 &3   
  \end{ytableau})=-t_2.$$
  
$ \mathbb{R}(\lambda/\mu, \alpha) =
\{
\begin{ytableau}
   1&2\\2   
\end{ytableau},
\begin{ytableau}
   1&2\\1 \\2   
\end{ytableau},
\begin{ytableau}
   1&2\\2&2   
\end{ytableau},\dots
 \}$

\vspace{0.3 cm}
$\mathbb{R}(\lambda/\mu, \beta)
=\{
\begin{ytableau}
   1&2\\2&3   
  \end{ytableau},
\begin{ytableau}
   1&2\\1&3\\2   
  \end{ytableau},
\begin{ytableau}
   1&2\\1&2\\2 &3  
\end{ytableau},\dots  
\}$

\vspace{0.3 cm}
Thus $RG_{\lambda/\mu}(\characx;\mathbf{t})=(RG_{(2,1)}(\characx;\mathbf{1}) + RG_{(2,1,1)}(\characx;\mathbf{1}) + RG_{(2,2)}(\characx;\mathbf{1}) + \cdots ) -t_2 (RG_{(2,2)}(\characx;\mathbf{1}) + RG_{(2,2,1)}(\characx;\mathbf{1}) + RG_{(2,2,2)}(\characx;\mathbf{1}) +\cdots )$.

\vspace{0.3 cm}
$ \mathbb{S}(\lambda/\mu, \alpha) = \{\begin{ytableau}
    1&2\\2   
\end{ytableau},
\begin{ytableau}
     1,2 &2   
\end{ytableau}
  \}$

\vspace{0.2cm}
$ \mathbb{S}(\lambda/\mu, \beta)=\{
\begin{ytableau}
   1&2\\2&3   
\end{ytableau},
\begin{ytableau}
   1&2,3\\2   
\end{ytableau}
\}$

\vspace{0.2cm}
\noindent Therefore, 
$RG_{\lambda / \mu}(\characx;\mathbf{t}) =
g_{(2,1)}(\characx)-g_{(2,0)}(\characx)- t_2 g_{(2,2)}(\characx) + t_2 g_{(2,1)}(\characx)$.
\end{example}
\subsection{Schur P-functions}
A partition $\lambda=(\lambda_1,\lambda_2,\dots,\lambda_l)$ is said to be \textit{strict} if $\lambda_1 > \lambda_2 > \cdots > \lambda_l >0.$
For a strict partition $\lambda,$ the \textit{shifted shape} $S(\lambda)$ of $\lambda$ is the array of boxes obtained by placing $\lambda_i$ boxes in the $i^{th}$ row, with each row shifted $i-1$ positions to the right with respect to the top row.

A word $u=u_1u_2 \cdots u_n$ in the alphabet $N=\{1 < 2 < \cdots \}$ is a \emph{hook word} if there exists a positive integer $1 \leq k \leq n$ such that $$ u_1>u_2>\cdots >u_k \leq u_{k+1} \leq \cdots \leq u_n .$$ 
\begin{definition}
A \textit{semi-standard decomposition tableau} (SSDT) \cite[Definition 2.14]{Serrano} is a filling $T$ of the shifted shape associated to the strict partition $\lambda=(\lambda_1,\dots,\lambda_l)$ with entries from $N$ such that  
\begin{itemize}
    \item the word $v_i$ formed by reading the $i^{th}$ row from left to right is a hook word of length $\lambda_i$, and 
    \item $v_i$ is a hook subword of maximum length in the concatenation $v_lv_{l-1}\cdots v_i$ for all $1 \leq i \leq l-1$.
\end{itemize}
The \emph{reading word} of $T$, denoted by $\Read(T),$ is defined by $v_lv_{l-1}\cdots v_1$ and the \emph{weight} $\wt(T)$ of $T$ is the weight of $\Read(T)$.   
\end{definition}
\begin{example}
It is easy to check that 
$$ T=
\begin{ytableau}
     3 &2 &1 &2\\
     \none & 2 & 1\\ 
     \none & \none & 2  
\end{ytableau}
$$ is a SSDT of shifted shape $(4,2,1)$ such that $\Read(T)=2213212$ and $\wt(T)=(2,4,1)$.   
\end{example}
\begin{definition}
For a strict partition $\lambda,$ the \textit{Schur P-function} $P_{\lambda}(x_1,x_2,\dots,x_n)$ is defined by $$P_{\lambda}(x_1,x_2,\dots,x_n):=\displaystyle\sum_{T \in \SSDT_n(\lambda)}\characx^{\wt(T)},$$
where $\SSDT_n(\lambda)$ be the set of all semi-standard decomposition tableaux of shifted shape $\lambda$ with entries at most $n$.
\end{definition}
\bremark
\label{def:crystal-SSDT}
For a strict partition $\lambda,$ we have the following embedding
$$\Read:\SSDT_n(\lambda) \rightarrow \mathbb{W}_n^{\otimes |\lambda|}, T \mapsto \Read(T),$$
where $\mathbb{W}_n$ is the standard type $A_{n-1}$ crystal in Example~\ref{ex:standard crystal}.

Using this embedding, we identify $\SSDT_n(\lambda)$ with a subset of $\mathbb{W}_n^{\otimes |\lambda|}$ and define the action of $e_i,f_i, \wt, \epsilon_i, \varphi$ on the elements of $\SSDT_n(\lambda)$. Then the set $\SSDT_n(\lambda)$ is a type $A_{n-1}$ crystal under these operators, see Remark~2.6 in \cite{Kashiwara:SSDT}.
\eremark
Every hook word is the row reading word of a unique semi-standard Young tableau whose shape is a hook partition. For example, the hook word $\textcolor{blue}{4}\textcolor{blue}{3}\textcolor{blue}{2}23$ is the row reading word of the semi-standard Young tableau below
$$\ytableausetup{mathmode,
notabloids}
\begin{ytableau}
     \textcolor{blue}{2} & 2 & 3 \\ \textcolor{blue}{3} \\ \textcolor{blue}{4}   
\end{ytableau}
$$
Let $l=l(\lambda).$ Then each SSDT $T$ of shifted shape $\lambda$ corresponds to a unique semi-standard tableau $\Tilde{T}=\Tilde{T_l} * \cdots *\Tilde{T_2} *\Tilde{T_1},$ where each $\Tilde{T_i}$ is the semi-standard Young tableau of hook partition shape such that $\Read(T_i)=r_{\Tilde{T_i}},$ where $T_i$ is the $i^{th}$ row of $T$.
\begin{example}
    \label{Example:SSDT}
Consider $ Q = \ytableausetup{mathmode,
notabloids}
\begin{ytableau}
    3 & 2 &1 & 1 \\ \none & 2 &1  
\end{ytableau}
 \in \SSDT_3(4,2)$. Then 
$$\Tilde{Q_1}= \ytableausetup{mathmode,
notabloids}
\begin{ytableau}
    1 &1\\2 \\3 
\end{ytableau} \hspace{1 cm} 
\Tilde{Q_2}= \ytableausetup{mathmode,
notabloids}
\begin{ytableau}
    1 \\ 2  
\end{ytableau} \implies \Tilde{Q}=\Tilde{Q_2} * \Tilde{Q_1}=
\begin{ytableau}
    \none & 1 & 1 \\ \none & 2 \\ \none & 3\\ 1 \\2  
\end{ytableau}
$$
\end{example}
Let $Q$ be a highest weight element in $\SSDT_n(\lambda)$ and $\Tilde{Q}=\Tilde{Q_l} * \cdots * \Tilde{Q_2} *\Tilde{Q_1}$. Now if $(
\begin{bmatrix}
      \mathbf{b}(\sh(\Tilde{Q})) \\
       r_{\Tilde{Q}} \\
\end{bmatrix}
\rightarrow \emptyset)=(\rect (\Tilde{Q}), \hat{Q})$ (Theorem~\ref{Theorem:Burge}), we call $\hat{Q}$ is a $\lambda$-compatible tableau for SSDT.
\begin{example}
We take the semi-standard decomposition tableau $Q \in \SSDT_3(4,2),$ which is also a highest weight element. Then
$$(
\begin{bmatrix}
           \mathbf{b}(\sh(\Tilde{Q})) \\
           r_{\Tilde{Q}} \\
\end{bmatrix} \rightarrow \emptyset)
=(\ytableausetup{mathmode,
notabloids}
  \begin{ytableau}
    1&1&1\\2 &2 \\3  
  \end{ytableau},\ytableausetup{mathmode,
notabloids}
  \begin{ytableau}
    1&1&4\\2 &5\\3
  \end{ytableau} ).$$
So $\hat{Q} = \ytableausetup{mathmode,
notabloids}
\begin{ytableau}
    1&1&4\\2 &5\\3
\end{ytableau}$ is a $(4,2)$-compatible tableau for SSDT.
\end{example}
Let $Q$ be any highest weight element of $\SSDT_n(\lambda)$ and $ \SSDT_n(\lambda;Q)$ be the connected component of the crystal graph of $\SSDT_n(\lambda)$ containing $Q$.
Also, let $n^n=(n,n,\dots,n) \in \mathbb{Z}^{n}_+$. Now we define the following map 
$$\Gamma :\SSDT_n(\lambda;Q) \rightarrow \mathcal{A}(\hat{Q}, \sh(\Tilde{Q}),n^n) \text{ by }
S \mapsto \Tilde{S} =\Tilde{S_l} * \cdots * \Tilde{S_2}*\Tilde{S_1}.$$
\begin{proposition}
\label{Proposition:main-SSDT}
The map $\Gamma : \SSDT_n(\lambda;Q)  \rightarrow \mathcal{A}(\hat{Q}, \sh(\hat{Q}),n^n)$ defined by $ \Gamma(S)=\Tilde{S}$ is a weight-preserving bijection which intertwines the crystal raising and lowering operators.
\end{proposition}
\begin{proof}
First we check that $\Gamma$ commutes with $e_i,f_i$ for each $i$. Let $h_i$ be either $e_i$ or $f_i$. We have to show that $\Gamma(h_i.T)=h_i\Gamma(T)$, i.e., $\widetilde{h_i.T}=h_i.\Tilde{T}$ which is equivalent to showing $ r_{\widetilde{h_i.T}}=r_{h_i.\Tilde{T}}$.
Now for a skew shape $\lambda/\mu (\lambda,\mu \in \mathcal{P}[n
]),$ the crystal structure on $\Tab_n(\lambda/\mu)$ is given by the following embedding into $\mathbb{W}_n^{\otimes|\lambda|-|\mu| }$:
$$ T \mapsto r_T=v_1v_2\cdots v_{|\lambda|-|\mu|} \mapsto v_1 \otimes v_2 \otimes \cdots \otimes v_{|\lambda|-|\mu|}.$$
Thus $r_{h_i.\Tilde{T}}=h_i.r_{\Tilde{T}} =h_i.\Read(T)$ (since $\Read(T)=r_{\Tilde{T}}$) $=\Read(h_i.T)$ (using \ref{def:crystal-SSDT}) $=r_{\widetilde{h_i.T}}$. Now, for any $T \in \SSDT_n(\lambda;Q),T=f_{i_1}^{k_1}f_{i_2}^{k_2}\cdots f_{i_t}^{k_t}(Q)$. Thus, using \cite[Proposition 29]{RS}, $\Gamma(T)=\Tilde{T}=f_{i_1}^{k_1}f_{i_2}^{k_2}\cdots f_{i_t}^{k_t}(\Tilde{Q}) \in \mathcal{A}(\hat{Q}, \sh(\hat{Q}),n^n)$. Hence the map $\Gamma$ is well-defined. It is clear that $\Gamma$ is weight preserving. Let $T,T' \in \SSDT_n(\lambda;Q)$ such that $\Tilde{T}=\Tilde{T'}$. Then $r_{\Tilde{T}}=r_{\Tilde{T'}} \implies \Read(T)=\Read(T') \implies T=T'$. Thus $\Gamma$ is injective. Let $\Tilde{R}=\Tilde{R_l} *\cdots *\Tilde{R_1} \in  \mathcal{A}(\hat{Q}, \sh(\hat{Q}),n^n)$. Then $\Tilde{R}=f_{t_1}^{s_1}\cdots f_{t_k}^{s_k}(\Tilde{Q})$. Since $\Read(Q)=r_{\Tilde{Q}}, R=f_{t_1}^{s_1} \cdots f_{t_k}^{s_k}(Q) \in \SSDT_n(\lambda;Q)$ and $\Gamma(R)=\Tilde{R}$. So $\Gamma$ is surjective.
\end{proof}
Let $\mathcal{T}(\lambda)$ indicate the set of all $\lambda$-compatible tableaux for SSDT.
Then the above proposition and Proposition~\ref{proposition:tableau} provides the following tableaux Schur expansion of Schur P-function 
\begin{equation}
 \label{eq:SSDT}
P_{\lambda}(x_1,x_2.\dots,x_n)=\displaystyle\sum_{Q \in \mathcal{T}(\lambda)}s_{\sh(Q)}
\end{equation}
Therefore, using Theorem~\ref{theorem:K bases}, we obtain the following: 
$$ 
P_{\lambda}(x_1,x_2,\dots,x_n)=\displaystyle\sum_{R\in \mathcal{R}(\lambda)}RG_{\sh(R)}(\characx,\mathbf{1})
$$
$$
\hspace{3.85 cm}=\displaystyle\sum_{S\in \mathcal{S}(\lambda)}(-1)^{|\ex(S)|}g_{\sh(S)}(\characx)
$$
Here $\mathcal{R}(\lambda)$ (resp. $\mathcal{S}(\lambda)$) denotes the set of all reverse plane partitions (resp. set-valued tableaux) whose reading word is Knuth equivalent to an element of $\mathcal{T}(\lambda)$.
\begin{example}
The only highest weight elements of $\SSDT_3(3,1)$ are the following:
$$ 
U=
\begin{ytableau}
    2 & 1 &1 \\ \none & 1
\end{ytableau}
\hspace{0.5 cm}
V= 
\begin{ytableau}
    2 & 1 &1 \\ \none & 2
\end{ytableau}
\hspace{0.5 cm}
W=
\begin{ytableau}
    3 & 2 & 1 \\ \none & 1
\end{ytableau}
$$
Then, 
$$ 
\Tilde{U}=
\begin{ytableau}
    \none & 1 &1 \\ \none & 2 \\  1
\end{ytableau} \implies
(
\begin{bmatrix}
           \mathbf{b}(\sh(\Tilde{U})) \\
           r_{\Tilde{U}} \\
\end{bmatrix} \rightarrow \emptyset)
=(\ytableausetup{mathmode,
notabloids}
  \begin{ytableau}
    1&1&1\\2   
  \end{ytableau},\ytableausetup{mathmode,
notabloids}
  \begin{ytableau}
    1&1&3\\2 
  \end{ytableau} ).
$$
$$
\Tilde{V}= 
\begin{ytableau}
   \none & 1 &1 \\ \none & 2 \\ 2
\end{ytableau} \implies (
\begin{bmatrix}
           \mathbf{b}(\sh(\Tilde{V})) \\
           r_{\Tilde{V}} \\
\end{bmatrix} \rightarrow \emptyset)
=(\ytableausetup{mathmode,
notabloids}
  \begin{ytableau}
    1&1\\2 &2  
  \end{ytableau},\ytableausetup{mathmode,
notabloids}
  \begin{ytableau}
    1&1\\2 & 3
  \end{ytableau} ).
$$
$$
\Tilde{W}=
\begin{ytableau}
    \none & 1 \\ \none & 2 \\ \none & 3 \\  1
\end{ytableau} \implies 
(
\begin{bmatrix}
           \mathbf{b}(\sh(\Tilde{W})) \\
           r_{\Tilde{W}} \\
\end{bmatrix} \rightarrow \emptyset)
=(\ytableausetup{mathmode,
notabloids}
  \begin{ytableau}
    1&1\\2 \\3  
  \end{ytableau},\ytableausetup{mathmode,
notabloids}
  \begin{ytableau}
    1&4\\2 \\3
  \end{ytableau} ).$$
Thus, $\mathcal{T}(\lambda)=\{\begin{ytableau}
    1&1&3\\2 
  \end{ytableau},  \begin{ytableau}
    1&1\\2 & 3
  \end{ytableau},  \begin{ytableau}
    1&4\\2 \\3
  \end{ytableau} \}.$ 
  
\vspace{0.1cm}
$\{
R \in \mathcal{R}(\lambda) | \rect(w(R))=\begin{ytableau}
    1&1&3\\2 
  \end{ytableau}
\}=
\{
\begin{ytableau}
   1&1&3\\2   
\end{ytableau},
\begin{ytableau}
   1&1&3\\1\\2   
  \end{ytableau},
\begin{ytableau}
   1&1&3\\2\\2   
  \end{ytableau},\dots
\}$
 
\vspace{0.2cm}
$\{
R \in \mathcal{R}(\lambda) | \rect(w(R))=\begin{ytableau}
    1&1\\2 & 3 \end{ytableau}
\}
=\{
\begin{ytableau}
   1&1\\2&3   
  \end{ytableau},
\begin{ytableau}
   1&1\\2&3\\2   
  \end{ytableau},
\begin{ytableau}
   1&1\\1&3\\2   
\end{ytableau},\dots  
\}$

\vspace{0.2cm}
$\{
R \in \mathcal{R}(\lambda) | \rect(w(R))=\begin{ytableau}
    1&4\\2 \\3
  \end{ytableau}
\}
=\{
\begin{ytableau}
    1&4\\2 \\3
\end{ytableau},  
\begin{ytableau}
   1&4\\1 \\2\\3   
\end{ytableau},  
\begin{ytableau}
   1&4\\2\\2\\3   
\end{ytableau},
\begin{ytableau}
   1&4\\2 \\3\\3   
\end{ytableau},\dots  
\}
$
\vspace{0.3 cm}

Thus
$P_{\lambda}(\characx)=(RG_{(3,1)}(\characx,\mathbf{1})+2RG_{(3,1,1)}(\characx,\mathbf{1})+\cdots) + (RG_{(2,2)}(\characx,\mathbf{1})+2RG_{(2,2,1)}(\characx,\mathbf{1})+\cdots)+(RG_{(2,1,1)}(\characx,\mathbf{1})+3RG_{(2,1,1,1)}(\characx,\mathbf{1})+\cdots)$.

\vspace{0.3 cm}
$\{
S \in \mathcal{S}(\lambda) | \rect(w(S))=\begin{ytableau}
    1&1&3\\2 
  \end{ytableau}
\} = \{\begin{ytableau}
    1&1&3\\2   
  \end{ytableau},
  \begin{ytableau}
    1& {1,2} &3   
  \end{ytableau}
  \}$

\vspace{0.2cm}
$\{
S \in \mathcal{S}(\lambda) | \rect(w(S))=\begin{ytableau}
    1&1\\2 & 3 \end{ytableau}
\}
=\{
\begin{ytableau}
   1&1\\2 &3   
\end{ytableau},
\begin{ytableau}
   1&1,3\\2   
\end{ytableau}
\}$

\vspace{0.2cm}
$\{
S \in \mathcal{S}(\lambda) | \rect(w(S))=\begin{ytableau}
    1&4\\2 \\3
  \end{ytableau}
\}
=\{\begin{ytableau}
  1&4\\2 \\3  
\end{ytableau},
\begin{ytableau}
  1&4\\2,3  
\end{ytableau},
\begin{ytableau}
  1,2 &4\\3  
\end{ytableau},
\ytableausetup{mathmode,
notabloids, boxsize=0.9cm}
\begin{ytableau}
  1,2,3&4  
\end{ytableau}
\}$

Therefore, 
$P_{\lambda}(\characx) = 
g_{(3,1)}(\characx)-g_{(3,0)}(\characx)+g_{(2,2)}(\characx)-g_{(2,1)}(\characx)+g_{(2,1,1)}(\characx)-2g_{(2,1)}(\characx)+g_{(2,0)}(\characx)$.
\end{example}
\section*{Acknowledgements}
This work was initiated during the author's PhD at IMSc Chennai and was completed while the author was working as a Post-doctoral Fellow at TIFR Mumbai. The work is subsequently revised during the author’s tenure at NISER, Bhubaneswar.

\end{document}